\documentclass[11pt,a4paper,reqno]{amsart}

\usepackage{times} % assumes new font selection scheme installed
\usepackage{amsmath} % assumes amsmath package installed
\usepackage{amssymb}  % assumes amsmath package installed
\usepackage{amsthm}
\usepackage{latexsym}
\usepackage{amsfonts,bbm,dsfont,bm}
\usepackage{xcolor}
\usepackage{mathtools}
\usepackage{enumerate}
\usepackage{cite}
\usepackage{url}
\usepackage{tikz}
\usepackage{pgfplots}
\pgfplotsset{compat=1.17}
\usepgfplotslibrary{groupplots}
\usepackage{graphicx,caption}
\usepackage{subcaption}
\usepackage{todonotes}
\usepackage{float}
\usepackage{nicefrac}
\usepackage{mathrsfs}

\usepackage{ifpdf}                         % ihr nutzt \ifpdf in main.tex
\usepackage{algorithm}                     % nur falls ihr algorithm-Umgebungen nutzt
\usepackage{hyperref}                      % vor cleveref!
\usepackage[capitalize,nameinlink]{cleveref} % als (fast) letztes Paket

\usetikzlibrary{arrows.meta,positioning,calc}

\newtheorem{theorem}{Theorem}[section]
\newtheorem{lemma}[theorem]{Lemma}
\newtheorem{proposition}[theorem]{Proposition}
\newtheorem{corollary}[theorem]{Corollary}
\theoremstyle{definition}
\newtheorem{definition}[theorem]{Definition}

\theoremstyle{remark}
\newtheorem{remark}[theorem]{Remark}

\crefname{section}{section}{sections}
\Crefname{section}{Section}{Sections}
\crefformat{equation}{(#2#1#3)}
\Crefformat{equation}{Equation~(#2#1#3)}

\usepackage[foot]{amsaddr}

\numberwithin{equation}{section}
\allowdisplaybreaks
\title[Extended dynamic mode decomposition with Fourier dictionaries]{Extended dynamic mode decomposition with Fourier dictionaries: Error bounds and fast implementation}
\author{Felix~Bartel$^{1}$ \and Sandra~Ritter$^{2,3}$ \and Manuel~Schaller$^{2}$ \and Karl~Worthmann$^{3}$}
\address{$^{1}$Mathematical Institute for Machine Learning and Data Science (MIDS), Catholic University of Eichst\"att-Ingolstadt, Germany}
\address{$^{2}$Faculty of Mathematics, Chemnitz University of Technology, Chemnitz, Germany}
\address{$^{3}$Optimization-based Control Group, Institute of Mathematics, Technische Universit\"at Ilmenau, Germany}
\date{\today}

\newcommand{\R}{\ensuremath{\mathbb R}}    % Reelle Zahlen
\newcommand{\C}{\ensuremath{\mathbb C}}    % Komplexe Zahlen
\newcommand{\N}{\ensuremath{\mathbb N}}    % Nat"urliche Zahlen
\newcommand{\Z}{\ensuremath{\mathbb Z}}    % Ganze Zahlen
\newcommand{\T}{\ensuremath{\mathds T}}    % Torus
\newcommand{\calK}{\ensuremath{\mathcal K}}  
\newcommand{\calO}{\ensuremath{\mathcal O}}

\begin{document} %%%%%%%%%%%%%%%%%%%%%%%%%%%%%%%%%%%%%%%%%%%%%%%%%%%%%%%%%%%%%%
	
	\begin{abstract}
		The Koopman operator has gained considerable attention in dynamical systems due to its capability to provide a linear viewpoint for nonlinear systems using data-driven methods such as extended dynamic mode decomposition (EDMD). 
        In this work, we propose an EDMD-variant with a Fourier dictionary on the $d$-dimensional torus, where the data are sampled on an equispaced tensor grid. 
		We show that the EDMD regression problem admits a unique closed-form solution and coincides with trigonometric interpolation of the Koopman image. This identification has two consequences. 
        First, using Koopman invariance of Sobolev spaces, we transfer approximation-theoretic results for trigonometric interpolation to derive $L^p$-error bounds, $2 \le p \le \infty$, on approximations of the Koopman operator. In particular, the established bounds are of optimal order with  explicit constants. Second, the EDMD matrix never has to be assembled, since its action reduces to (nonequispaced) fast Fourier transforms such that the EDMD-approximation can be evaluated matrix-free with quasi-linear cost in the dictionary size. We illustrate the results for the Kuramoto model of coupled oscillators and for the chaotic ABC flow with dictionaries of up to $10^7$ modes. 
	\end{abstract}


	\maketitle

	\section{Introduction}\label{sec: intro}

	The Koopman operator lifts a nonlinear dynamical system to a linear, but infinite-dimensional one acting on the space of observable functions \cite{koopman1931hamiltonian,mezic2005spectral}. Since linear objects are amenable to spectral analysis~\cite{rowley2009spectral} along with tailored stable long-term prediction and control~\cite{StraWort26}, this viewpoint has become a central tool in data-driven modeling and analysis of complex dynamical systems, see the surveys~\cite{brunton2022modern,mezic2026koopman} and the collection \cite{mauroy2020koopman}. Extended dynamic mode decomposition (EDMD) \cite{williams2015data,kutz2016dynamic} approximates the action of the operator on a finite-dimensional dictionary of observables from data by least squares.
    Numerically robust algorithms for data-driven Koopman approximations have been developed in \cite{Drmac2018modal}.
    However, the cost of EDMD usually grows at least quadratically in the dictionary size $N$, since two matrices consisting of dictionary evaluations have to be formed. 
    Since the approximation is exact only on Koopman-invariant subspaces, which a finite dictionary in general does not span \cite{Brunton2016invariant}, large dictionaries are needed in higher state dimensions or if the Koopman image develops fine-scale structures, as is the case for chaotic dynamics.

	In this paper, we show that this bottleneck can be removed by combining a Fourier dictionary of bandwidth $n$ on the torus with samples on an equispaced grid. Then, the EDMD regression problem has a unique closed-form solution, and the inverse appearing therein is, up to scaling, the adjoint of the $d$-dimensional discrete Fourier matrix. Hence, the EDMD matrix never has to be assembled and its action is evaluated by the fast Fourier transform (FFT) \cite{CT65,potts2018numerical} and, for the nonequispaced image points of the dynamics, by the nonequispaced FFT (NFFT) \cite{potts2018numerical,keiner2009nfft,barnett2019parallel}, i.e., with quasi-linear cost and linear storage in $N$. Moreover, the EDMD matrix turns out to represent trigonometric interpolation of the Koopman image, which is the key to our error analysis.
	
	In the infinite-data limit, EDMD converges to the $L^2$-orthogonal projection of the Koopman operator onto the dictionary \cite{korda2018convergence}, whereas quantitative error bounds are comparatively recent. For kernel EDMD \cite{williams2015kernel,klus2020eigendecompositions}, deterministic $L^\infty$-bounds in terms of the fill distance were established in \cite{philipp2024error,KohnPhil24} and in \cite{HertPhil25} for stochastic dynamics, and $L^2$-type finite-data bounds were derived in \cite{Mezi22:approximations,nuske2023finite,philipp2026variance} for a fixed dictionary and in \cite{zhang2024quantitative} for finite-element dictionaries. All of these bounds rest on approximation results for the respective dictionary. Correspondingly, our analysis builds upon trigonometric interpolation, which approximates functions of Sobolev smoothness $s$ with an $L^p$-error of order $(n^d)^{-\frac sd+\frac 12-\frac 1p}$, $2\le p\le\infty$, i.e., with the best possible rate among all algorithms using this number of point evaluations, see \cite{Temlyakov18,dung2018hyperbolic} and \cite{ullrich2014approximation} for sharp constants of the underlying Sobolev embeddings. In contrast to kernel methods, it adapts automatically to the regularity of the observable, and since the fractional Sobolev spaces on the torus coincide \emph{isometrically} with interpolation spaces \cite{chandler2015interpolation}, no additional constants enter for non-integer smoothness.

    Fourier techniques have been used in Koopman theory before, but with different objectives. 
    For Koopman spectral analysis, Fourier bases and the FFT are employed in \cite{govindarajan2019approximation} to approximate the spectra of measure-preserving maps on the torus, and in \cite{Drmac2019DFT} the discrete Fourier transform of snapshots along a single trajectory yields a numerically robust computation of Koopman eigenvalues and modes.
    Moreover, \cite{nuske2023efficient} uses random Fourier features \cite{rahimi2007random} to reduce the cost of kernel EDMD. In contrast, we exploit a deterministic Fourier dictionary together with its equispaced grid, which enables both FFT-based computation and a priori error bounds.
	
	\textbf{Contribution.} The contributions of this paper are as follows. (i) We derive the closed-form solution of EDMD with Fourier dictionaries on equispaced grids (\Cref{cor: Fourier EDMD matrix}), identify it with trigonometric interpolation (\Cref{prop:edmd-interpolation}) and evaluate the resulting two Koopman approximants matrix-free by means of the FFT and the NFFT. (ii) In \Cref{prop:error bound}, we prove $L^p$-error bounds of optimal order for both approximants with constants that are explicit in the map $F$ and its smoothness. (iii) For the Kuramoto model on $\T^d$, $d\le 5$, with dictionaries of up to almost $10^7$ modes and for the chaotic ABC flow on $\T^3$, we validate the predicted rates in numerical experiments with a runtime that grows quasi-linearly in $N$. 
    
	\textbf{Outline.} The paper is organized as follows. In \Cref{sec: Fourier EDMD}, we recall EDMD and derive the closed-form solution for Fourier dictionaries. In \Cref{sec: EDMD as interpolation}, we identify it with trigonometric interpolation and introduce the two Koopman approximants, which are analyzed in \Cref{sec: error analysis}. \Cref{sec: numerics} describes the matrix-free implementation and presents the results for the Kuramoto model and the ABC flow. Conclusions are drawn in \Cref{sec: conclusion}.

	\section{Extended dynamic mode decomposition with Fourier dictionaries}\label{sec: Fourier EDMD}
	
	Throughout, let $\T^d=\R^d/\Z^d$ be the $d$-dimensional torus represented in the Euclidean space by $[0,1]^d$ with opposite faces identified.
	We consider a discrete time dynamical system $x^+ = F(x)$ with $F : \mathds T^d\to \mathds T^d$ a continuous, potentially nonlinear map.
	Then $F$ induces the linear \emph{Koopman operator} $\mathcal{K}$ mapping functions on $\T^d$ to functions on $\T^d$ via
	\begin{equation}\label{eq:Koopman}
		\mathcal{K} f := f \circ F,\qquad f:\T^d \to \C.
	\end{equation}
    Thus, the Koopman operator propagates observables rather than states, yielding linear dynmacis even if $F$ is nonlinear. The tradeoff is that the Koopman-lifted system is, in general, infinite-dimensional.

    \subsection{EDMD}
	We briefly recall \textit{Extended Dynamic Mode Decomposition} (EDMD;  \cite{williams2015data}), a widely-used data-driven method to approximate the Koopman operator.
	Assume that we are given $M \in\mathbb{N}$ data snapshots $(x_i,y_i)$ with $y_i=F(x_i)$, $i=1,...,M$, and define the sets $X=\{x_1,...,x_M\}$ and $Y=\{y_1,...,y_M\}$.
	Choosing a set of linearly independent observables $\psi_k:\T^d \to \C$, $k=1,...,N$, we define the finite-dimensional space, often referred to as \textit{dictionary}, by $\mathbb{V}_N \coloneqq \operatorname{span}\{\psi_1,...,\psi_N\}$.
    From the identity~\eqref{eq:Koopman}, we obtain $\mathcal{K}\psi_k(x_j) = \psi_k(F(x_j)) = \psi_k(y_j)$.
	Hence, EDMD seeks a matrix $K\in\mathbb{C}^{N \times N}$ such that 
	\begin{align}\label{eq:approx1}
		\Psi_Y \approx K\Psi_X
	\end{align}
    where $\Psi(x) = (\psi_1(x),...,\psi_N(x))^\top$ collects the evaluation of the dictionary functions at the data snapshots and $\Psi_X=\left[\Psi(x_1),...,\Psi(x_M)\right]$, $\Psi_Y=\left[\Psi(y_1),...,\Psi(y_M)\right]\in\mathbb{C}^{N\times M}$.
	The approximation \eqref{eq:approx1} is then performed by solving the least squares problem 
	\begin{equation} \label{eq:LS}
		\min_{K\in\mathbb{C}^{N \times N}} \|K\Psi_X - \Psi_Y\|_F^2 = \min_{K\in\mathbb{C}^{N \times N}} \sum_{i=1}^M \|K\Psi(x_i) - \Psi(y_i)\|_2^2. 
	\end{equation}
	Here, the Frobenius norm corresponds to the empirical $L^2$-norm induced by the sampled data. Consequently, the minimizer yields the best least-squares fit of the Koopman action on the dictionary at the sample points.
	Assuming that $\Psi_X$ has full row rank, the solution is given by
	\begin{equation}\label{eq:standard_EDMD}
		K = \Psi_Y\Psi_X^\dagger = \Psi_Y\Psi_X^\ast(\Psi_X\Psi_X^\ast)^{-1}.
	\end{equation}
	% is the matrix representation of the discrete projection $P^X_N\mathcal{K}|_{\mathbb{V}_N}$. In the infinite-data limit, $M\to\infty$, convergence of $P^X_N\mathcal{K}|_{\mathbb{V}_N}$ to the $L^2$-orthogonal projection $P_N\mathcal{K}|_{\mathbb{V}_N}$ is shown in \cite{korda2018convergence}.
	This matrix $K$ represents a Monte-Carlo-type data-based approximation of the action of the Koopman operator restricted to the dictionary basis $\{\psi_1,...,\psi_N\}$. As shown in \cite{korda2018convergence}, under suitable assumptions, this approximation converges to the $L^2$-orthogonal projection $P_N\mathcal{K}|_{\mathbb{V}_N}$ as the number of data points tends to infinity.
	% and the finite-dimensional approximation $\widehat\calK: \mathbb{V}_N \to \mathbb{V}_N$ of the Koopman operator is then defined by 
	% \begin{equation*}
		%     \widehat\calK \psi = c^{\mathsf H} K \Psi
		% \end{equation*}
	% for any observable $\psi = c^{\mathsf H} \Psi \in \mathbb{V}_N$ in the dictionary. 

	\subsection{EDMD with Fourier dictionaries on equispaced grids} 
    
	In this work, we show how Fourier dictionaries are particularly well-suited for fast and provably optimal approximations in terms of the number of samples.
	For a multi-index $ k=(k_1,\ldots,k_d)^\top \in\mathbb Z^d$ we define the Fourier basis functions
	\begin{equation}\label{eq: Fourier basis}
		\psi_{ k}( x) = \exp\!\big( 2\pi \mathrm i \langle  k,  x\rangle \big), \qquad  x\in\mathds T^d.
	\end{equation}
	where $\langle  k,  x\rangle\coloneqq\sum_{j=1}^d k_j x_j$.
    The family $\{\psi_k\}_{k\in\Z^d}$ forms an orthonormal basis of $L^2$.
	For a fixed bandwidth $n\in\mathbb N$ we consider the finite dictionary
	\begin{equation}\label{eq:Fdict}
		\mathbb{V}_n = \mathrm{span}\{ \psi_{ k} \;:\;  k\in K_n\},
		\qquad 
		K_n := \{  k\in\mathbb Z^d \;:\; \| k\|_\infty \le n \}.
	\end{equation}
	Thus, we have $\dim(\mathbb V_n)=(2n+1)^d \eqqcolon N$.
	To compute the approximation via EDMD, we consider a tensor-product equispaced sampling grid given by
	\begin{equation}\label{eq: grid d-dim}
		 X = \left\{  x_{ j} = \left(\frac{j_1}{2n+1},...,\frac{j_d}{2n+1}\right)           \;:\;  j_1,...,j_d=0,\ldots,2n \right\} \;\subset\;\mathds T^d.
	\end{equation}
	Therein, the points in $ X$ are indexed by multi-indices $ j\in \mathcal I_n$ with $\mathcal{I}_n := \{0,\ldots,2n\}^d$.
	Importantly, the number of grid points matches the dimension of the Fourier dictionary, that is, $|X| = \dim(\mathbb{V}_n)=N$. 
	Given data snapshots $(x_j,F(x_j))$ for $ x_j\in X$, $j\in \mathcal I_n$
	and collecting the evaluation of the dictionary functions at the data snapshots in the $N\times N$~matrices 
	\begin{equation*}
		( L_{ X})_{ k, j} = \psi_{ k}( x_{ j})
		\qquad \text{and} \qquad
		( L_{ X^+})_{ k,  j} = \psi_{ k}(F( x_{ j})),
	\end{equation*}
	where rows are indexed by frequencies $ k\in K_n$ and columns by data points $ j\in \mathcal{I}_n$, the EDMD least squares problem \eqref{eq:LS} becomes
    $\min_{ K\in\mathbb{C}^{N \times N}} \|K L_{ X} -  L_{ X^+}\|_F^2$.
    Due to the particular structure of the dictionary and the sample points, the matrix $L_X$ has favorable properties allowing for fast inversion that lay the foundation for fast Fourier methods.
    This is a consequence of the orthogonality of the Fourier basis on the equispaced grid: by \cite[Lemma 4.66]{potts2018numerical}, we have for every $ m\in\mathds Z^d$
	\begin{equation}\label{eq: character}
		\frac{1}{N}\sum_{ j\in \mathcal{I}_n} \psi_{ m}( x_{ j})
		= \begin{cases}
			1, &  m\in (2n+1)\mathds Z^d ,\\
			0, & \text{otherwise.}
		\end{cases}
	\end{equation}

	\begin{lemma}\label{lem: Fourier orthogonality}
		Let $ X$ be the uniform tensor-product grid \eqref{eq: grid d-dim}. Then $L_{ X}  L_{ X}^\ast~=~N\, I_{N}$ such that $ L_{ X}$ is invertible with
        $L_{ X}^{-1} = \frac{1}{N}\,  L_{ X}^\ast$.
	\end{lemma}

	\begin{proof}
		For $ k, \ell\in K_n$ we have $( L_{ X} L_{ X}^\ast)_{ k, \ell} = \sum_{ j\in \mathcal{I}_n}\psi_{ k- \ell}( x_{ j})$, and $\| k- \ell\|_\infty\le 2n < 2n+1$ implies that $ k- \ell\in (2n+1)\mathds Z^d$ if and only if $ k= \ell$.  Hence \eqref{eq: character} yields $ L_{ X} L_{ X}^\ast = N I_{N}$.
	\end{proof}
	
	Consequently, using a Fourier dictionary and equispaced points, the inverse entering the solution of the EDMD regression problem is known analytically and therefore never has to be computed. In particular, it has a very simple form and is expressed only using the data matrices $L_X$ and~$L_{X^+}$. 
	
	\begin{corollary}\label{cor: Fourier EDMD matrix}
		Let $\mathbb{V}_n$ be the Fourier dictionary and let $ X$ be the equispaced tensor-product grid defined in \eqref{eq: grid d-dim}.
		Denote by $ L_{ X},  L_{ X^+}$ the corresponding Fourier evaluation matrices.
		Then the EDMD least-squares problem admits the unique solution
		\begin{equation}\label{eq:FKoop}
			 K =  L_{ X^+}  L_{ X}^{-1}
			= \frac{1}{N} \,  L_{ X^+}  L_{ X}^\ast .
		\end{equation}
	\end{corollary}
	
	The explicit representation of the EDMD matrix $ K$ obtained in \Cref{cor: Fourier EDMD matrix} is the key structural property of Fourier-EDMD on equispaced grids. 
    It has two fundamental consequences that constitute the core of this work.
	
	First, $L_X^{-1} = \tfrac{1}{N} L^*_X$ admits an interpretation as reconstruction of trigonometric polynomials from their values on the sampling grid. In particular, the operator associated with $K$ can be interpreted as the composition which maps sampled data to reconstructed Fourier representations after transport by the dynamics $F$. As we show in \Cref{sec: EDMD as interpolation}, this leads to an exact characterization of Fourier-EDMD in terms of trigonometric interpolation, which allows us to transfer classical results from approximation theory to the analysis of EDMD. 
	To be precise, the matrix $ K$ induces a linear operator
	\begin{equation*}
		\calK_n : \mathbb{V}_n \to \mathbb{V}_n , \qquad (\calK_n f)(x) = ( K^\top\widetilde f)^\top \Psi(x) = \widetilde f \,^\top  K \Psi(x)
	\end{equation*}
	where $\widetilde f$ denotes the coefficient vector of $f\in\mathbb{V}_n$ and $\Psi( x) := (\psi_{ k}( x))_{ k\in K_n}$, i.e., $f(x) = \widetilde f \,^\top \Psi(x)$.
	The operator $\calK_n$ represents the finite-dimensional Fourier-EDMD approximation of the Koopman operator $\calK$ on the space $\mathbb{V}_n$. 
	In \Cref{sec: error analysis} it will be shown that, under suitable Sobolev regularity assumptions, the approximation error satisfies an error bound of the form
	\begin{equation*}
		\|(\calK - \calK_n) f\|_{L^p} \le C(n^d)^{-\frac sd+\frac 12-\frac 1p} \|f\|_{H^s}, \quad \forall f\in\mathbb{V}_n
	\end{equation*}
	where the constant $C>0$ is explicitly known. Using a suitable projection from the right, we will further show that a similar estimate also holds for functions in $H^s$.
	Here, we stress that the rate $-\frac sd+\frac 12-\frac 1p$ is optimal in the worst-case setting for observables with Sobolev smoothness $s$ and $(2n+1)^d$ point evaluations.
	
	Second, the matrix $L_X$ possesses a discrete Fourier transform structure. Consequently, both $L_X$ and $L_X^{-1}$ and hence the EDMD approximation $ K$ can be applied by means of multidimensional FFTs and NFFTs, allowing for highly-efficient matrix-free implementations with quasi-linear complexity $\mathcal{O}(N\log N)$ with the total number of points and dictionary size $N$ that enables computations for high-dimensional dictionaries, as illustrated in \Cref{sec: numerics}.
	
	\section{EDMD as trigonometric interpolation} \label{sec: EDMD as interpolation}
	
	By \Cref{cor: Fourier EDMD matrix} the Fourier-EDMD approximation is completely determined by the Fourier sampling matrices $L_X$ and $L_{X^+}$. 
	The purpose of this section is to show that Fourier-EDMD on equispaced grids coincides with trigonometric interpolation of Koopman images. This identification forms the basis for the error analysis performed in \Cref{sec: error analysis}.
	
	\subsection{Trigonometric interpolation on equispaced grids}

    We first recall the trigonometric interpolation operator associated with the equispaced grid $X$.
	
	\begin{definition} \label{def: trigonometric interpolation}
		Let $ X$ be the equispaced tensor-product grid \eqref{eq: grid d-dim}.
		The trigonometric interpolation operator $S_n : C\to \mathbb V_n$ is defined by assigning to every $f\in C$ the unique trigonometric polynomial $S_nf\in\mathbb V_n$ satisfying $(S_nf)( x_{ j}) = f( x_{ j})$ for all $j\in\mathcal I_n$.
	\end{definition}
	To relate the trigonometric interpolation operator $S_n$ with the matrix $L_X$ involved in the EDMD surrogate, observe that, by definition, every $p\in\mathbb V_n$ admits a unique representation $p( x) = c^\top \Psi( x)$ with a coefficient vector $c=(c_{ k})_{ k\in K_n}\in\mathbb C^{N}$.
	For a continuous function $f\in C$ we denote by $f_X = \bigl( f( x_{ j}) \bigr)_{ j\in\mathcal I_n} \in\mathbb C^{N}$ the vector of its values on the sampling grid.
	
	\begin{lemma}\label{lem: interpolation coefficients}
		Let $f\in C$ and let $S_nf( x) = \widetilde f \,^\top\Psi( x)\in\mathbb V_n$ be the trigonometric interpolant of $f$ with coefficient vector $\widetilde f\in\mathbb C^{N}$. 
		Then, the coefficients satisfy
        $\widetilde f \,^\top = \frac1{N}\, f_X^\top L_X^\ast$.
		In particular, $S_nf$ admits the representation $S_nf( x) = \frac1{N} f_X^\top L_X^\ast\Psi( x)$.
	\end{lemma}
	
	\begin{proof}
		For every $ j\in\mathcal I_n$ we have $S_nf( x_{ j}) = \widetilde f \,^\top\Psi( x_{ j}) = \widetilde f \,^\top (L_X)_{:, j}$, since the $ j$-th column of $L_X$ is $\Psi( x_{ j})$.
		Collecting these values for all $ j\in\mathcal I_n$ yields $(S_nf)_X^\top = \widetilde f \,^\top L_X$.
		Moreover, since $S_nf$ interpolates $f$ on the grid $ X$, we have $f_X^\top = (S_nf)_X^\top = \widetilde f \,^\top L_X$ and hence $\widetilde f \,^\top = f_X^\top (L_X)^{-1} = \frac1{N}\, f_X^\top L_X^\ast$, where the second equality holds by \Cref{lem: Fourier orthogonality}. 
	\end{proof}
	
	The representation above provides an explicit characterization of trigonometric interpolation in terms of the Fourier evaluation matrix. The matrix $L_X^{-1}$ maps function values on the grid to Fourier coefficients, while the feature vector $\Psi(x)$ evaluates the resulting trigonometric polynomial at arbitrary points. Next, we show that the EDMD approximation obtained in \Cref{sec: Fourier EDMD} coincides with trigonometric interpolation applied to Koopman images.

	\subsection{The EDMD matrix as trigonometric interpolation}
    
    First, we observe that the EDMD approximant preserves the fundamental property of the Koopman operator propagating observable functions along the flow.
	\begin{lemma}\label{lem:feature-propagation}
		For every sampling point $ x_{ j}\in X$, $j\in \mathcal I_n$, $K\Psi( x_{ j}) = \Psi(F( x_{ j}))$.
	\end{lemma}
	
	\begin{proof}
		By \Cref{cor: Fourier EDMD matrix}, we have $ KL_X=L_{X^+}$ such that the result follows.
	\end{proof}
	\Cref{lem:feature-propagation} shows that the EDMD matrix reproduces the action of the dynamical system on the feature vectors at the sampling points. 
    In other words, it is the discrete analogue of the Koopman identity $(\mathcal Kf)(x) = f(F(x))$.
	Consequently, the EDMD approximation may be interpreted as propagating Fourier feature vectors through the dynamics and subsequently reconstructing a trigonometric polynomial from the propagated data.
	
	\begin{proposition}\label{prop:edmd-interpolation}
		Let $f(x)=\widetilde f\,^\top\Psi(x)\in\mathbb V_n$.
		Then $(S_n\mathcal Kf)(x) = \widetilde f\,^\top  K\Psi(x)$, i.e. the matrix $ K$ represents the operator $S_n \mathcal K|_{\mathbb V_n}$ w.r.t. the Fourier basis of \ $\mathbb V_n$, see \eqref{eq:Fdict}.
	\end{proposition}
		
		\begin{proof}
			Let $f\in\mathbb{V}_n$ with $f( x)=\widetilde f \,^\top\Psi( x)$.
			By \Cref{lem:feature-propagation}, we have $\widetilde f\,^\top  K\Psi( x_{ j}) = \widetilde f\,^\top\Psi(F( x_{ j})) = f(F( x_{ j}))$, for every sampling point $ x_{ j}\in X$.
			Hence the trigonometric polynomial $p(x):=\widetilde f\,^\top K\Psi(x)$ interpolates the function $f\circ F$ on the grid $ X$.
			Since $p\in\mathbb V_n$ and as the trigonometric interpolant is unique, we have $p = S_n(f\circ F) = S_n\mathcal Kf$.
		\end{proof}
		
		\Cref{prop:edmd-interpolation} shows that EDMD with the particular choice of a Fourier dictionary and an equispaced grid is not merely a least-squares approximation procedure. On the dictionary space $\mathbb V_n$, it coincides with trigonometric interpolation of the Koopman image of the dictionary.
		This interpretation naturally extends to arbitrary continuous observables and leads to two Koopman approximants, depending on the available data. These approximants were already introduced and discussed in the context of kernel EDMD in \cite{KohnPhil24} and \cite[Remark~2.3]{bold2025kernel}.
		Firstly, if the values $f(F( x_{ j}))$ are available we may consider the Koopman approximant
		\begin{equation}\label{eq: one-sided interpolation}
			\widehat{\mathcal{K}}_1 \coloneqq S_n\mathcal{K}: C \to \mathbb{V}_n
		\end{equation}
		which maps an observable $f\in C$ to the trigonometric interpolant of $f\circ F$, i.e.,
        \begin{align*}
            \widehat{\mathcal{K}}_1 f = S_n (f\circ F) = \frac{1}{N} (f\circ F)_X^\top L_X^* \Psi.
        \end{align*}    
        This approximant may be implemented using the fast Fourier transform.
    
		If only the values on the \emph{original} sampling grid $X$ are available, one must first replace $f$ by its trigonometric interpolant. This leads to the data-driven approximation
		\begin{equation}\label{eq: two-sided interpolation}
			\widehat{\mathcal{K}}_2 \coloneqq S_n\mathcal{K}S_n: C \to \mathbb{V}_n.
		\end{equation}
		Then, the Koopman approximant \eqref{eq: two-sided interpolation} applied to $f$ is
		\begin{equation*}
			\widehat{\mathcal{K}}_2f 
			= S_n\mathcal{K}S_nf 
			= S_n\mathcal{K}|_{\mathbb{V}_n}(S_nf) 
			= \widetilde{(S_nf)}^\top  K \Psi
			= \widetilde{f}^\top  K \Psi
			% = \frac{1}{(2n+1)^d}\tilde{f}\,^\top  L_{ X^+} L_{ X}^\ast \Psi 
			= \frac{1}{N^2}f_{ X}^\top  L_{ X}^\ast L_{ X^+} L_{ X}^\ast \Psi.
		\end{equation*}
		Note that the application of this operator only requires the values of $f$ on the grid $ X$, together with the evaluation of the dictionary functions at the image points $F( x_{ j})$, but does not require direct evaluation of $f$ at off-grid locations.
		While the multiplication with $ L_{ X}^\ast$ can be carried out via FFT, the multiplication with $ L_{ X^+}$ involves evaluations of a trigonometric polynomial at the nonuniform points $F( x_{ j})$ and therefore requires an extension of the FFT to nonequispaced points (NFFT).
		
		\section{Error analysis in Sobolev spaces}\label{sec: error analysis}

        The equivalence of Fourier-EDMD and trigonometric interpolation of Koopman images established in the previous section implies that the EDMD approximation error in Sobolev spaces is governed by two ingredients. First, the approximation properties of trigonometric interpolation, and, second, the invariance of Sobolev spaces under the Koopman operator together with a bound on its operator norm. In this section, we quantify both and combine them into an error bound with explicit constants.

        \subsection{Sobolev spaces on the torus} 
        
        We first introduce Sobolev spaces on the torus used throughout this work. 
		For $f\in L^2$ let $\hat f_{ k} = \langle f,\psi_k \rangle_{L^2} = \int_{\mathds T^d} f( x)\,\exp(-2\pi \mathrm i\langle k, x\rangle)\,\mathrm d x$, $ k\in\Z^d$, denote its Fourier coefficients, i.e. $f = \sum_{k\in\Z^d} \hat{f}_k\psi_k$.  
        For $s\ge 0$ we define the \emph{fractional} Sobolev space 
		\begin{equation}\label{eq:Hs}
			H^s
			\coloneqq
			\left\{f\in L^2 \; : \; \|f\|_{H^s}<\infty\right\} \quad \text{with} \quad \|f\|_{H^s}^2 \; \coloneqq\; \sum_{ k\in\mathbb Z^d} (1+\| k\|_2^2)^s\,|\hat f_{ k}|^2.
		\end{equation}
		In particular, $H^0=L^2$ isometrically.
		It is well-known that for integer smoothness \(\sigma\in\mathbb N\), this norm is equivalent to the derivative-based Sobolev norm
		%\begin{equation*}
			$\|f\|_{H^\sigma_{D}}^2
			= \sum_{|\alpha|\le \sigma} \|\partial^{\alpha}f\|_{L^2}^2$ with $\partial^{\alpha}f
			= \frac{\partial^{|\alpha|}f} {\partial x_1^{\alpha_1}\cdots\partial x_d^{\alpha_d}}$
		%\end{equation*}
        and $|\alpha|\coloneqq\alpha_1+\dots+\alpha_d$ for a multi-index $\alpha_\in\N_0^d$.
		This norm equivalence, with explicit dimension-independent constants, is established for the torus $[0,2\pi]^d$ in \cite[eq.~(2.3)]{ullrich2014approximation} with a slightly different upper bound. We restate it in the scaling of the torus used in this work to explicitly control constants. 
			\begin{lemma}\label{lem:sobolev-norm-equiv}
				Let $\sigma\in\mathbb N$. Then, $\frac{1}{\sqrt{\sigma!}}\,\|f\|_{H^\sigma} \le \|f\|_{H^\sigma_D} \le (2\pi)^\sigma\,\|f\|_{H^\sigma}$ for all $f\in H^\sigma$.
			\end{lemma}
			\begin{proof}
				We follow the proof of \cite[eq.~(2.3)]{ullrich2014approximation}. By the multinomial identity, we have
				\begin{equation}\label{eq: multinomial identity}
					\big(1+\| k\|_2^2\big)^\sigma = \sum_{|\alpha|\le\sigma} \underbrace{\frac{\sigma!}{(\sigma-|\alpha|)!\,\alpha_1!\cdots\alpha_d!}}_{\eqqcolon c_\alpha} \ \prod_{i=1}^d k_i^{2\alpha_i} ,
				\end{equation}
				with $1 \hspace{-0.3mm} \le \hspace{-0.3mm} c_\alpha \hspace{-0.3mm} \le \hspace{-0.3mm} \sigma!$. Since $\partial^{\alpha}\psi_{ k} \hspace{-0.5mm} = \hspace{-0.5mm} \prod_{i=1}^d (2\pi\mathrm i k_i)^{\alpha_i}\psi_{ k}$, Parseval's identity gives 
                $\|f\|_{H^\sigma_D}^2 \hspace{-1.mm} = \hspace{-0.5mm} \sum_{ k\in\Z^d} w( k)|\hat f_{ k}|^2$ with $w( k) \coloneqq \sum_{|\alpha|\le\sigma}\prod_{i=1}^d (2\pi k_i)^{2\alpha_i}.$
                That is, $w(k)$ is the right-hand side of \eqref{eq: multinomial identity} with $2\pi k$ in place of $k$ and all coefficients $c_\alpha$ replaced by one.
                Since all summands are nonnegative, comparing the two sums termwise gives $w(k) \le (1+4\pi^2\| k\|_2^2)^\sigma \le \sigma!\,w( k)$.
                Together with $1+\| k\|_2^2 \le 1+4\pi^2\| k\|_2^2 \le 4\pi^2 (1+\| k\|_2^2)$, this yields 
                $\frac{1}{\sigma!} (1+\| k\|_2^2)^\sigma \le w(k) \le (2\pi)^{2\sigma}(1+\| k\|_2^2)^\sigma$ and taking square roots proves the claim.
			\end{proof}

        \subsection{Boundedness of the Koopman operator}
        
		The following theorem uses the norm equivalence established above to transfer known bounds for the Koopman operator in the derivative-based Sobolev norm to the Fourier-based Sobolev norm \eqref{eq:Hs}.
        To this end, for a multi-index $\gamma\in\N_0^d$ the set $\mathcal P(\gamma)$ contains all tuples $\beta=(\beta_1,\ldots,\beta_{|\gamma|})\in\{1,\ldots,d\}^{|\gamma|}$ in which every index $i$ occurs exactly $\gamma_i$ times, so that $\partial^\gamma=\partial_{\beta_1}\cdots\partial_{\beta_{|\gamma|}}$ for each such $\beta$.
        Further, $\Pi_m$ denotes the set of partitions $\pi$ of $\{1,\ldots,m\}$ into nonempty blocks and $|\pi|$ the number of blocks of a partition $\pi\in\Pi_m$. 
    
		\begin{theorem}\label{thm:koopman-well-defined and bound}
			Let $\sigma \in \mathbb{N}_0$ and assume that the dynamics satisfies $F \in C^\sigma$ and is a $C^1$-diffeomorphism with 
            $c_0^{-1}\coloneqq\inf_{ x\in\mathds T^d} |\det DF( x)| > 0$.
            Then, the Koopman operator $\mathcal{K} : H^\sigma \to H^\sigma$ is well defined and bounded with
            \begin{equation*}
                \| \mathcal{K}\|_{H^{\sigma} \to H^{\sigma}}  
                \le (2\pi)^\sigma\sqrt{\sigma!}\ C_F, 
                \qquad C_F \coloneqq \bigg(\max\bigg\{c_0,\sum_{1\le\|\alpha\|_1\le\sigma}S_{\alpha}(F)\bigg\}\bigg)^{1/2} \quad \text{and}
            \end{equation*}
            \begin{equation*}
                S_\alpha(F) \hspace{-0.5mm} \coloneqq \hspace{-0.5mm}\sup_{x\in\mathds{T}^d}|\det DF(x)|^{-1}\sum_{k=1}^{|\alpha|}\sum_{|\alpha'|=k}\Bigg|\sum_{\substack{\pi\in\Pi_{|\alpha|}\\|\pi|=k}}\sum_{\beta'\in\mathcal{P}(\alpha')} \hspace{-1.mm} \partial_{\beta,B_1}F_{\beta'_1}(x)\cdots\partial_{\beta,B_k}F_{\beta'_k}(x)\Bigg|^2
            \end{equation*}
            for $\alpha\in\mathbb{N}_0^d$, $1\le|\alpha|\le\sigma$, and any $\beta\in\mathcal{P}(\alpha)$ where $B_i$ denotes the $i$-th block of the partition $\pi$, and $\partial_{\beta,B}$ stands for the operator $\partial_{\beta_{j_1}}\cdots\partial_{\beta_{j_\ell}}$, where $B = \{j_1,\ldots,j_\ell\}$.
		\end{theorem}
		
		\begin{proof}
			In \cite[Theorem 4.2]{KohnPhil24}, well-definedness as well as boundedness are shown for the Koopman operator in the derivative-based norm, that is, the bound $\|\calK f\|_{H^\sigma_{D}} \leq C_F \|f\|_{H^\sigma_{D}}$.
            The norm equivalence of \Cref{lem:sobolev-norm-equiv} implies for the Fourier-based norm
			\begin{equation*}
				\|\calK f\|_{H^\sigma}
				\le \sqrt{\sigma!} \ \|\calK f\|_{H^\sigma_{D}}
				\le \sqrt{\sigma!} \ C_F \|f\|_{H^\sigma_{D}}
				\le (2\pi)^{\sigma}\sqrt{\sigma!} \ C_F \|f\|_{H^\sigma}, 
			\end{equation*}
            which is the asserted bound.
		\end{proof}
        
		The previous result establishes boundedness of the Koopman operator in $H^\sigma$ for integer smoothness $\sigma$. For $\sigma=0$ the sum in the definition of the constant $C_F$ is empty and the bound reads $\|\calK\|_{L^2\to L^2}\le\sqrt{c_0}$, which follows from the change of variables $ y = F( x)$.
	
        As in \cite[Theorem 4.2]{KohnPhil24} we use interpolation theory of Hilbert spaces (see e.g.\ \cite{chandler2015interpolation}), to extend these estimates to fractional-order Sobolev spaces.
		On the torus, the Fourier characterization of Sobolev spaces allows for a particularly simple description as fractional Sobolev spaces coincide \emph{isometrically} with the interpolation spaces between $L^2$ and $H^\sigma$. In the following result, $(\cdot,\cdot)_{\theta,2}$ denotes the real interpolation space obtained by the $K$-method with the normalized norm of \cite[Section 2]{chandler2015interpolation}.
		
		\begin{lemma}\label{lem: interpolation equal norm}
			Let $\sigma\in\mathbb{N}$ and $s\in(0,\sigma)$. Then, for $\theta = s/ \sigma \in(0,1)$ we have
			%\begin{equation}\label{eq: H^s interpolation space}
				$H^s = (L^2, H^\sigma)_{\theta,2}$ %\ , \quad \theta = \frac{s}{\sigma} \in (0,1)
			%\end{equation}
			with equality of norms.
		\end{lemma}
		\begin{proof}
			Consider the measure space $(\Z^d, \mathcal P(\Z^d), \mu)$ with counting measure $\mu$. 
			Denote by the mapping $\mathcal F$, $f\mapsto(\hat f_k)_{k\in\Z^d}$ the Fourier transform and set $w_0(k) \equiv 1$ and $w_1(k)=(1+\|k\|_2^2)^\sigma$.
			By Parseval's identity $\|f\|_{L^2}^2 = \sum_{k\in\Z^d} |\widehat{f}_k|^2 = \|\mathcal F f\|^2_{L^2(\Z^d, w_0\mu)}$ and similarly, by the definition \eqref{eq:Hs},
            $\|f\|_{H^\sigma}^2 = \sum_{k\in\Z^d} (1+\|k\|_2^2)^\sigma |\widehat{f}_k|^2 = \|\mathcal F f\|^2_{L^2(\Z^d, w_1\mu)}$.
			In particular, the mappings $\mathcal{F}:L^2 \to L^2(\Z^d, w_0\mu)$ as well as $\mathcal{F}: H^\sigma \to L^2(\Z^d, w_1\mu)$ are unitary isomorphisms. 
			Hence, the result follows from \cite[Corollary 3.2]{chandler2015interpolation}. 
		\end{proof}
		
		\begin{remark}
			For general domains $\Omega\subset\R^d$, the identification of Sobolev spaces with interpolation spaces holds only up to equivalence of norms. 
			To be precise, under suitable regularity assumptions on $\Omega$, e.g. bounded Lipschitz domains, one has $H^s(\Omega) = (L^2(\Omega), H^\sigma(\Omega))_{s/\sigma,2}$ with equivalent norms. However, the equivalence constants depend on the domains and associated extension operators and are, in general, unknown. 
		\end{remark}
		
		The isometric identification of fractional Sobolev spaces on the torus with interpolation spaces enables us to transfer the operator bounds in \Cref{thm:koopman-well-defined and bound} to arbitrary smoothness without introducing additional interpolation constants. 
		
		\begin{proposition}\label{prop: koopman bound non-integer s}
			Let $\sigma\in\mathbb{N}$ and assume that $F\in C^\sigma$ is a $C^1$-diffeomorphism satisfying $c_0^{-1}=\inf_{ x\in\mathds T^d} |\det DF( x)| > 0$. Then for all $s\in(0,\sigma]$ the linear Koopman operator $\mathcal{K} : H^s \to H^s$ is well-defined and bounded. In particular, we have
			\begin{equation*}
				\|\mathcal{K}\|_{H^s\to H^s} \leq % \ %A(s) 
                 (2\pi)^s \ c_0^{\frac{\sigma-s}{2\sigma}} \ (\sigma!)^{\frac{s}{2\sigma}} \ C_F^{\,s/\sigma} 
			\end{equation*}
            with $C_F$ as defined in \Cref{thm:koopman-well-defined and bound}.
		\end{proposition}
		\begin{proof}
			Let $\sigma\in\mathbb{N}$ and $s\in(0,\sigma)$.
			By \Cref{lem: interpolation equal norm}, $H^s = (L^2, H^\sigma)_{s/\sigma,2}$, with equality of norms.
			Since, by \Cref{thm:koopman-well-defined and bound}, the Koopman operator $\mathcal{K}$ is well-defined and bounded on $H^\sigma$ and on $L^2$, \cite[Lemma 22.3]{tartar2007introduction} implies that also $\mathcal{K}: H^s \to H^s$ is well-defined and bounded with $\|\mathcal{K}\|_{H^s\to H^s} \leq \|\mathcal{K}\|_{L^2\to L^2}^{1-s/\sigma} \|\mathcal{K}\|_{H^\sigma\to H^\sigma}^{s/\sigma}$.
			Substituting the bounds from \Cref{thm:koopman-well-defined and bound} yields the result.
		\end{proof}

        \subsection{The trigonometric interpolation error}
        
		Having established boundedness of the Koopman operator in Sobolev spaces, we now turn to the approximation error induced by the trigonometric interpolation.
		We first prove the following auxiliary lemma.

		\begin{lemma}\label{epstein} 
			For $n\in\mathbb N_0$, $d\in\mathbb N$, and $s>d/2$ it holds
			\begin{equation*}
				\sum_{ k\in\mathbb Z^d : \|k\|_\infty > n} \| k\|_2^{-2s} 
				\le
                \frac{4^d}{2s-d}
                \begin{cases}
                    2s & n=0, \\ 
                    n^{d-2s} & n\in\mathbb N .
                \end{cases}
			\end{equation*}
		\end{lemma} 
		\begin{proof} 
            We rewrite the sum using nesting $\ell_\infty$-shells
            \begin{align*}
                \sum_{\|k\|_\infty>n}\!\!\!\! \|k\|_2^{-2s}
                \le \!\!\!\sum_{\|k\|_\infty>n} \!\!\!\|k\|_\infty^{-2s}
                = \!\!\!\sum_{R=n+1}^{\infty} \!\!\! R^{-2s} \!\!\!\!\!\sum_{\|k\|_\infty=R}\!\!\!\! 1 
                = \!\!\!\sum_{R=n+1}^{\infty}\!\! R^{-2s}\Big((2R+1)^d-(2R-1)^d\Big).
            \end{align*}
            By the mean value theorem, $(2R+1)^d-(2R-1)^d \le 2d(2R+1)^{d-1} \le 2d 3^{d-1} R^{d-1} \le 4^{d} R^{d-1}$ and together with the above, we get $\sum_{\|k\|_\infty>n} \|k\|_2^{-2s} \le 4^{d}\sum_{R=n+1}^{\infty} R^{-2s+d-1}$.
            Set $\beta\coloneqq 2s-d+1$, i.e. that the exponent of the last sum equals $-\beta$, and note that $\beta>1$ by the assumption $s>d/2$.  Since $x\mapsto x^{-\beta}$ is monotone decreasing on $[1,\infty)$, we have $R^{-\beta}\le\int_{R-1}^{R}x^{-\beta}\,\mathrm dx$ for all $R\ge 2$, and the intervals $[R-1,R]$, $R\ge2$, cover $[1,\infty)$. 
            Hence, summing the above and keeping the summand $R=1$ separate gives
			\begin{align*}
				\sum_{ k\in\Z^d\setminus\{ 0\}} \| k\|_2^{-2s} 
				&\le 4^{d} \Big(1+\int_{1}^{\infty} x^{-\beta} \;\mathrm dx\Big)
				= 4^{d} \Big(1+\frac{1}{\beta-1}\Big)
				= 4^{d} \Big(1+\frac{1}{2s-d}\Big) \\
				&= 4^d\,\frac{2s-d+1}{2s-d}
				\le \frac{4^d}{2s-d} 2s \,,
			\end{align*}
            where the last inequality holds since $2s-d+1\le 2s$ for $d\ge 1$.\\
            The case $n\in\mathbb N$ is analogous without the need of separating the summand $R=1$.
		\end{proof}
		
		The above lemma together with the Fourier characterization of Sobolev spaces leads to the following approximation estimate for trigonometric interpolation. While the convergence rate is classical, see e.g. \cite[Theorem~3.6.4]{Temlyakov18}, we require an estimate with fully explicit constants, since it is our aim to derive explicit error bounds for the Fourier-EDMD approximation.
		We note that the following error bound on the interpolation is optimal in the sense that no algorithm using $(2n+1)^d$ point evaluations attains a better rate in the worst case over $H^s$, see \cite{Temlyakov18,dung2018hyperbolic}. 
        
		\begin{theorem}\label{thm: bound interpolation error} 
			Let $d,n\in\mathbb N$, $s>d/2$, and $0\le r \le s$ be smoothness parameters and $S_n$ the $d$-dimensional trigonometric interpolation from \Cref{def: trigonometric interpolation} based on $(2n+1)^d$ point evaluations.
			Then
			\begin{equation*}
				\|\operatorname{Id}-S_n\|_{H^s\to H^r}
				\le 2^{d+1}\sqrt{\frac{s}{2s-d}}(2d)^{r/2}\ (n^d)^{\frac{r-s}{d}}  \,.
			\end{equation*}
			In particular, the trigonometric interpolation operator is bounded from $H^s$ into itself, i.e.
			%\begin{equation*}
				$\|S_n\|_{H^s\to H^s}
				\le 2^{d+2}\sqrt{\frac{s}{2s-d}}(2d)^{s/2}$.
			%\end{equation*}
		\end{theorem} 

		\begin{proof} 
			Using Parseval's identity, we decompose the error
			\begin{equation}\label{eq: Parseval decomp}
				\|f-S_n f\|_{H^r}^{2}
				= \|f-P_n f\|_{H^r}^{2} + \|P_n f - S_n f\|_{H^r}^{2} \,,
			\end{equation}
			where $P_n f = \sum_{\| k\|_\infty\le n} \hat f_{ k}\exp(2\pi\mathrm i\langle k,\cdot\rangle)$  is the $L^2$-orthogonal projection.
			The first summand is the error due to projection, which we bound, using $\| k\|_2\ge\| k\|_\infty\ge n+1$ and $r-s\le 0$, by
			\begin{align*}
				&\|f-P_n f\|_{H^r}^{2}
				= \sum_{\| k\|_\infty > n} (1+\| k\|_2^2)^{r-s} (1+\| k\|_2^2)^s |\hat f_{ k}|^2 \\
				&\le (1+(n+1)^2)^{r-s} \sum_{\| k\|_\infty > n} (1+\| k\|_2^2)^s |\hat f_{ k}|^2 
				\le n^{2(r-s)} \|f\|_{H^s}^{2}  \le (2d)^r n^{2(r-s)} \|f\|_{H^s}^{2} \,.
			\end{align*}
		The second summand of \eqref{eq: Parseval decomp} is due to aliasing effects of the approximation $S_n f$.
		By \Cref{lem: interpolation coefficients}, its coefficients are $\widetilde f_{ k} = \frac{1}{N}\sum_{ j\in\mathcal I_n} f( x_{ j})\overline{\psi_{ k}( x_{ j})}$, $ k\in K_n$.
			Since $s>d/2$, the Fourier series of $f$ converges absolutely, so that we may insert it and interchange the order of summation, 
            $$\widetilde f_{ k}
				= \sum_{ m\in\mathbb Z^d} \hat f_{ m}\,
				\frac{1}{N}\sum_{ j\in\mathcal I_n} \psi_{ m- k}( x_{ j})
				= \sum_{ \ell\in\mathbb Z^d} \hat f_{ k+(2n+1) \ell},$$
			by \eqref{eq: character}, since $ m- k\in (2n+1)\mathbb Z^d$ if and only if $ m= k+ (2n+1) \ell$ for some $ \ell\in\mathbb Z^d$.
			Hence
			\begin{equation}\label{eq:aliasing}
				S_n f = \sum_{\| k\|_\infty\le n} \Big( \sum_{ \ell\in\mathbb Z^d} \hat f_{ k+(2n+1) \ell} \Big) \psi_{ k} \,,
			\end{equation}
			that is, the grid cannot separate frequencies which differ by an integer multiple of $(2n+1)$ in every component, and the $ k$-th coefficient of the interpolant collects all of them.

			Now we are able to bound the aliasing error.
			By \eqref{eq:aliasing} and $1+\| k\|_2^2\le 1+dn^2\le 2dn^2$ for $\| k\|_\infty\le n$ we have
            \begin{align*}
				&\|P_n f-S_n f\|_{H^r}^2
				= \sum_{\| k\|_\infty\le n} (1+\| k\|_2^2)^r \Big| \sum_{ \ell\in\mathbb Z^d\setminus\{ 0\}} \hat f_{ k+(2n+1) \ell} \Big|^2 \\
				&\le \hspace{-0.5mm} (2d)^r n^{2r} \hspace{-2.5mm}\sum_{\| k\|_\infty\le n} \hspace{-0.5mm}\Big| \hspace{-0.5mm} \sum_{ \ell\in\mathbb Z^d\setminus\{ 0\}} \hspace{-1.5mm}
				(1\hspace{-0.5mm}+\hspace{-0.5mm}\| k\hspace{-0.5mm}+\hspace{-0.5mm}(2n\hspace{-0.5mm}+\hspace{-0.5mm}1) \ell\|_2^2)^{-\frac s2} (1\hspace{-0.5mm}+\hspace{-0.5mm}\| k\hspace{-0.5mm}+\hspace{-0.5mm}(2n\hspace{-0.5mm}+\hspace{-0.5mm}1) \ell\|_2^2)^{\frac s2} \hat f_{ k+(2n+1) \ell} \Big|^2 \,.
			\end{align*}
			Applying Cauchy--Schwarz inequality yields
			\begin{align*}
				\|P_n f-S_n f\|_{H^r}^2 
				\le (2d)^r n^{2r} \!\!\sum_{\| k\|_\infty\le n}
				&\Big(\sum_{ \ell\in\mathbb Z^d\setminus\{ 0\}}
				\!\!(1+\| k+(2n+1) \ell\|_2^2)^{-s}\Big) \\
				&\Big(\sum_{ \ell'\in\mathbb Z^d\setminus\{ 0\}}
				\!\!(1+\| k+(2n+1) \ell'\|_2^2)^{s}
				|\hat f_{ k+(2n+1) \ell'}|^2\Big) \,.
			\end{align*}
			For the first inner sum we note that for $\| k\|_\infty\le n$ and $\ell\neq 0$ every component satisfies $|k_i+(2n+1)\ell_i|\ge n|\ell_i|$.
            Hence, $\| k+(2n+1)\ell\|_2\ge n\|\ell\|_2$ and
			$(1+\| k+(2n+1) \ell\|_2^2)^{-s} \le \|n\ell\|_2^{-2s} = n^{-2s}\|\ell\|_2^{-2s}$.
			Thus we are able to apply \Cref{epstein} and obtain
			\begin{align}
				&\|P_n f-S_n f\|_{H^r}^2 \nonumber\\
				&\le
				\frac{2^{2d+1}s}{2s-d}\,
				(2d)^r n^{2r} n^{-2s} \sum_{\| k\|_\infty\le n}
				\sum_{\ell'\in\mathbb Z^d\setminus\{ 0\}}
				(1+\| k+(2n+1) \ell'\|_2^2)^{s}
				|\hat f_{ k+(2n+1) \ell'}|^2 \nonumber\\
				&\le \frac{2^{2d+1}s}{2s-d}(2d)^r n^{2(r-s)}  \|f\|_{H^s}^2 \,, \label{eq:secondsummand}
			\end{align}
            where the last step uses that $( k, \ell)\mapsto  k+(2n+1) \ell$ is a bijection from $K_n\times\mathbb Z^d$ onto $\mathbb Z^d$, such that the sets $\{ k+(2n+1) \ell: \ell\neq 0\}$, $ k\in K_n$, are pairwise disjoint.
			Combining the inequalities for the projection error and the aliasing error yields
			\begin{equation*}
				\|\operatorname{Id}-S_n\|_{H^s\to H^r}^2
				\le \Big(1+\frac{2^{2d+1}s}{2s-d}\Big)(2d)^{r} n^{2(r-s)}
				\le \frac{2^{2d+2}s}{2s-d}(2d)^{r} n^{2(r-s)}  \, ,
			\end{equation*}
            where we used $\frac{2^{2d+1}s}{2s-d}\ge 1$ in the last step.
			The bound on $S_n$ follows from
            $\|S_n\|_{H^s\to H^s}\le 1+\|\operatorname{Id}-S_n\|_{H^s\to H^s}$ and the case $r=s$, since $2^{d+1}\sqrt{\frac{s}{2s-d}(2d)^s}\ge 1$.
		\end{proof} 

        Note, that the trigonometric interpolation operator $S_n$ does not utilize the smoothness information, while the rate in the upper bound improves for higher smoothness.
		Hence, Fourier-based interpolation is universal in the sense that the error bounds adapt to the smoothness of the target function $f$ without changing the implementation.

        We can extend the $L^2$-error bound to $L^p$ for $2\le p\le \infty$ by separately bounding the error in $L^\infty$ and applying an interpolation argument.
        \begin{theorem}\label{thm: bound interpolation error Lp}
			Let $2\le p\le\infty$, $d,n\in\mathbb N$, and $s>d/2$ be the smoothness parameter and $S_n$ the $d$-dimensional trigonometric interpolation from \Cref{def: trigonometric interpolation} based on $(2n+1)^d$ point evaluations.
			Then
            \begin{equation*}
                \sup_{\|f\|_{H^s}\le 1} \|f-S_n f\|_{L^p}
                \le 2^{2d+2}\sqrt{\frac{s}{2s-d}} (n^d)^{-\frac sd + \frac 12-\frac 1p}.
            \end{equation*}
        \end{theorem}
        \begin{proof}
            We investigate the $L^\infty$ error first and obtain intermediate bounds using interpolation.
            By triangle inequality, we have
            \begin{equation}\label{eq:decomp}
                \|f-S_nf\|_{L^\infty} \le \|f-P_n f\|_{L^\infty} + \|P_n f- S_n f\|_{L^\infty} .
            \end{equation}
            For the first summand we obtain
            \begin{equation*}
                \|f-P_n f\|_{L^\infty}
                = \Big\| \sum_{\|k\|_\infty > n} \hat f_k \psi_k \Big\|_{L^\infty} 
                \le \sum_{\|k\|_\infty > n} |\hat f_k| \|k\|_2^{-s}(1+\|k\|_2^2)^{\frac s2} .
            \end{equation*}
            Applying Cauchy--Schwarz inequality and using \Cref{epstein} yields
            \begin{align*}
                \|f\hspace{-0.5mm}-\hspace{-0.5mm}P_n f\|_{L^\infty}
                \hspace{-0.5mm}\le \hspace{-0.5mm}\Big(\sum_{\|k\|_\infty > n} \hspace{-2.mm} |\hat f_k|^2 (1\hspace{-0.5mm}+\hspace{-0.5mm}\|k\|_2^2)^{s}\Big)^{\frac12}
                \Big( \sum_{\|k\|_\infty > n} \hspace{-2.mm}\|k\|_2^{-2s} \Big)^{\frac12} 
                \hspace{-0.5mm}\le\hspace{-0.5mm} \Big( \frac{4^d}{2s-d} n^{d-2s}\Big)^{\frac12}\|f\|_{H^s} .
            \end{align*}
            For any function $g=\sum_{\|k\|_\infty\le n}\hat g_k\psi_k$ with frequencies on the cube, we have by Cauchy--Schwarz inequality
            \begin{align*}
                \|g\|_{L^\infty}
                % &= \Big\| \sum_{\|k\|_\infty\le n} \hat g_k\psi_k \Big\|_{L^\infty} \\
                &\le \sum_{\|k\|_\infty\le n} |\hat g_k|
                \le (2n+1)^{\frac d2} \Big(\sum_{\|k\|_\infty\le n} |\hat g_k|^2\Big)^{\frac 12}
                = (2n+1)^{\frac d2} \|g\|_{L^2}.
            \end{align*}
            Applying this to the second summand of \eqref{eq:decomp} we revert the error back to the $L^2$ case, which we have estimated in \eqref{eq:secondsummand}
            \begin{equation*}
                \|P_n f-S_n f\|_{L^\infty}
                \le (2n+1)^{\frac d2} \|P_n f-S_n f\|_{L^2}
                \le \Big(\frac{4^{2d} 2s}{2s-d} \Big)^{\frac 12} n^{-s+\frac d2} \|f\|_{H^s}.
            \end{equation*}
            Combining both estimates, we obtain
            \begin{align*}
                \|f-S_nf\|_{L^\infty}
                \hspace{-0.5mm}\le\hspace{-0.5mm} \Big(\Big( \frac{4^d}{2s-d} \Big)^{\frac12}
                \hspace{-1.5mm}+\hspace{-0.5mm} \Big(\frac{4^{2d} 2s}{2s-d} \Big)^{\frac12}\Big)
                n^{-s+\frac d2}
                \|f\|_{H^s}
                \hspace{-0.5mm}\le\hspace{-0.5mm} \Big( \frac{s 4^{2d+2}}{2s-d} \Big)^{\frac12}
                \hspace{-0.5mm}n^{-s+\frac d2}
                \|f\|_{H^s}.
            \end{align*}
            Now we interpolate between the $L^2$ case of \Cref{thm: bound interpolation error} and the $L^\infty$ error bound above.
            For any $g\in L^\infty$ we have
            \begin{equation*}
                \|g\|_{L^p}^p
                = \int_{\mathds T^d} |g|^p\;\mathrm d\mu
                = \int_{\mathds T^d} |g|^2 |g|^{p-2}\;\mathrm d\mu
                \le \|g\|_{L^\infty}^{p-2} \int_{\mathds T^d} |g|^2 \;\mathrm d\mu
                = \|g\|_{L^\infty}^{p-2} \|g\|_{L^2}^{2} .
            \end{equation*}
            It is left to take the $p$-th root.
            By the assumption $s>d/2$ we have $H^s\hookrightarrow C \subset L^\infty$, and it is thus legitimate to apply the interpolation argument to the residual.
            We obtain
            \begin{equation*}
                \|f-S_n f\|_{L^p}
                \le \|f-S_n f\|_{L^2}^{\frac 2p}\|f-S_n f\|_{L^\infty}^{1-\frac 2p}
                \le \Big(\frac{4^{2d+2}s}{2s-d}\Big)^{\frac 12} n^{-s + d(\frac 12-\frac 1p)}\|f\|_{H^s},
            \end{equation*}
            which concludes the proof.
        \end{proof}
        The above result holds for Sobolev spaces $H_q^s$ with $1<q<\infty$, where $H_q^s$ is defined by replacing the $L^2$-norm in \eqref{eq:Hs} by $L^q$.
        Further, they also hold for the de la Vallée Poussin kernels instead of the Dirichlet kernels realizing the interpolation, cf.~\cite[Theorem~3.6.4]{Temlyakov18}.

		\subsection{Error bounds for the Koopman approximants}
        
		We are now prepared to state a main result of this work. Combining the approximation properties of trigonometric interpolation with the boundedness of the Koopman operator yields an $L^2$ error bound for the Fourier-EDMD approximation. 

		\begin{theorem}[Error bound for Koopman approximation]\label{prop:error bound}
			Let $2\le p\le\infty$, $\sigma\in\mathbb{N}$, $\sigma>d/2$ and assume $F\in~C^\sigma$ is a $C^1$-diffeomorphism with $c_0^{-1}\hspace{-0.5mm}\coloneqq\inf_{ x\in\mathds T^d} |\det DF(x)|\hspace{-0.5mm}>\hspace{-0.5mm}0$. Then for all $s\in(d/2,\sigma]$, the approximation $\widehat{\mathcal{K}}_1$ defined in \eqref{eq: one-sided interpolation} satisfies the error bound
			\begin{align*}
				\|\mathcal{K}-\widehat{\mathcal{K}}_1\|_{H^s\to L^p} \le C  (n^d)^{-\frac sd+\frac 12-\frac 1p}
			\end{align*}
			with $C \coloneqq 2^{2d+2}\sqrt{\frac{s}{2s-d}}\  A(s)$, $\ A(s) = (2\pi)^s \ c_0^{\frac{\sigma-s}{2\sigma}} \ (\sigma!)^{\frac{s}{2\sigma}} \ C_F^{\,s/\sigma}$.
			The approximant $\widehat{\mathcal{K}}_2$, defined in \eqref{eq: two-sided interpolation}, satisfies with $\widetilde C \coloneqq 2^{2d+2} c_0^{1/p} \sqrt{\frac{s}{2s-d}} + 2^{4d+4} \frac{s}{2s-d} (2d)^{s/2} A(s)$ the bound
			\begin{align*}
				\|\mathcal{K}-\widehat{\mathcal{K}}_2\|_{H^s\to L^p} \le \widetilde C  (n^d)^{-\frac sd+\frac 12-\frac 1p}.
			\end{align*}
		\end{theorem}
		\begin{proof}
			The proof follows \cite[Theorem 3.4]{KohnPhil24}.
			For the approximant $\widehat{\mathcal{K}}_1=S_n\mathcal{K}$, we have $\mathcal{K} - \widehat{\mathcal{K}}_1 = \mathcal{K} - S_n \mathcal{K} = (\operatorname{Id}-S_n)\mathcal{K}$.
			Substituting the bounds from \Cref{prop: koopman bound non-integer s} and \Cref{thm: bound interpolation error Lp} yields
			\begin{align*}
				\| \mathcal{K} - \widehat{\mathcal{K}}_1 \|_{H^s\to L^p}
				&\leq \| \operatorname{Id}-S_n \|_{H^s\to L^p} \| \mathcal{K} \|_{H^s\to H^s}
				\leq 2^{2d+2} \sqrt{\frac{s}{2s-d}}\ A(s)\ n^{-s+d(\frac 12-\frac 1p)} \ .
			\end{align*}
			Similarly, for $\widehat{\mathcal{K}}_2=S_n\mathcal{K}S_n$, we have $\mathcal{K} - \widehat{\mathcal{K}}_2 = \mathcal{K} - S_n \mathcal{K} S_n = \mathcal{K}(\operatorname{Id}-S_n) + (\operatorname{Id}-S_n)\mathcal{K}S_n$.
			Therefore, the approximation error satisfies
			\begin{align*}
				\| \mathcal{K} - \widehat{\mathcal{K}}_2 \|_{H^s\to L^p}
				&\leq \| \mathcal{K}(\operatorname{Id}-S_n) \|_{H^s\to L^p} + \| (\operatorname{Id}-S_n)\mathcal{K}S_n \|_{H^s\to L^p} \\
				&\leq \| \mathcal{K} \|_{L^p\to L^p} \| \operatorname{Id}-S_n \|_{H^s\to L^p} \\
                &\hphantom{==} + \| \operatorname{Id}-S_n \|_{H^s\to L^p} \| \mathcal{K} \|_{H^s\to H^s} \| S_n \|_{H^s\to H^s} \\
				&= (\| \mathcal{K} \|_{L^p\to L^p} + \| \mathcal{K} \|_{H^s\to H^s} \| S_n \|_{H^s\to H^s}) \cdot \| \operatorname{Id}-S_n \|_{H^s\to L^p}.
			\end{align*}
			The claim follows since $\|\calK\|_{L^p\to L^p}\le c_0^{1/p}$ for $p\in (1,\infty)$, $\|\calK\|_{L^\infty\to L^\infty}=1$, together with the interpolation estimate from \Cref{thm: bound interpolation error Lp} and uniform boundedness of $S_n$ and $\mathcal{K}$ in \(H^s\).
		\end{proof}

		We compare this result with error bounds for kernel-based approximation.
		\begin{remark}[Comparison with fill-distance-based kEDMD bounds]\label{rem: kEDMD comparison}
			The error bounds of \Cref{prop:error bound} rely crucially on the equispaced structure of the data set underlying trigonometric interpolation and thus differ from error bounds available for kernel EDMD (kEDMD), which apply to arbitrary data sets.
			In \cite{KohnPhil24}, deterministic $L^\infty$-error bounds for kEDMD are derived by identifying the regression problem with kernel interpolation in a reproducing kernel Hilbert space (RKHS) of Wendland functions, and bounding the resulting error in terms of the \emph{fill distance}
            $h_{X} \coloneqq \sup_{ x\in\mathds T^d} \min_{ x_i\in X} \| x -  x_i\|_2$ of an arbitrary data set $X=\{ x_1,\dots, x_M\}$. For Wendland kernels of smoothness order $\tau\in\N_0$, the resulting convergence rate is of order $h_X^{\tau-d/2}$.
			For our equidistant grid of $(2n+1)^d$ points, the fill distance satisfies $h_X\asymp n^{-1}$, so that a rate of $n^{-s}$ corresponds to the fill-distance rate $h_X^{s}$, rendering the error bound of \Cref{prop:error bound} structurally similar to \cite{KohnPhil24} in terms of the rate.
		\end{remark}

		\section{Numerical examples} \label{sec: numerics}

         We illustrate the error bounds of \Cref{prop:error bound} as well as the computational cost of the implementation on two examples: the Kuramoto map on $\T^d$ for $d\le5$, where the dynamics is given in closed form, and the chaotic ABC flow on $\T^3$, where the flow map must be approximated numerically. 
        The Python code and the data underlying all figures are available at 
        \begin{center}
            \url{mytuc.org/ysss}.
        \end{center}
        
        \subsection{Kuramoto model on $\mathds{T}^d$}\label{sec: Kuramoto}
		
		As a scalable test case for Fourier-EDMD we use the \emph{Kuramoto model} of coupled oscillators \cite{kuramoto1975international}.
		The model describes a system of $d$ coupled oscillators, each characterized by a phase $\theta_i \in \R / 2\pi\Z$, $i=1,\dots,d$ and an intrinsic natural frequency $\omega_i \in \mathbb{R}$, the rate at which it would rotate in the absence of coupling.
		The classical, all-to-all coupled Kuramoto model with \emph{coupling strength} $\kappa \ge 0$ is given by
		\begin{equation}
			\dot\theta_i = \omega_i + \frac{\kappa}{d}\sum_{j=1}^d \sin(\theta_j - \theta_i),
			\qquad i = 1,\dots,d,
			\label{eq:kuramoto-ode}
		\end{equation}
		where the $1/d$ normalization of the coupling term is the standard scaling that keeps the right-hand side well behaved as $d\to\infty$. 
		Rescaling $x_i = \theta_i/(2\pi)\in\T=\mathbb{R}/\mathbb{Z}$ and applying one explicit-Euler step with step size $h>0$ yields a discrete-time map $F:\T^d\to\T^d$,
		\begin{equation}\label{eq:kuramoto-map}
			F_i(x) = x_i + \nu_i + \frac{h\kappa}{2\pi d}\sum_{j=1}^{d}\sin\!\big(2\pi(x_j-x_i)\big) \pmod 1 ,
			\qquad i=1,\dots,d.
		\end{equation}
        The parameter $\nu_i\coloneqq \frac{h\omega_i}{2\pi}$ describes the phase increment of the $i$-th uncoupled oscillator per time step.
		For $d=1$ the coupling sum vanishes and
		\eqref{eq:kuramoto-map} reduces to a rigid rotation.
		For the numerical implementation, the coupling sum is evaluated using the equivalent identity
		\begin{equation}\label{eq: Kuramoto coupling sum}
			\sum_{j=1}^d \sin\!\big(2\pi(x_j-x_i)\big) = \cos(2\pi x_i) \sum_{j=1}^d \sin\!\big(2\pi x_j\big) - \sin\!\big(2\pi x_i \big)\sum_{j=1}^d \cos\!\big(2\pi x_j\big),
		\end{equation}
		replacing the $d\times d$ coupling matrix by two scalar sums per sampling point. This reduces both work and memory of evaluation of $F$ from $\mathcal O(d^2)$ to $\mathcal O(d)$.

        We verify that $F$ satisfies the assumptions of \Cref{prop:error bound}. Clearly, $F\in C^\infty(\T^d,\T^d)$. Writing $F( x) =  x + \nu + h\kappa g( x) \bmod 1$ with $g_i( x) = \frac{1}{2\pi d}\sum_{j=1}^d \sin(2\pi(x_j-x_i))$ for  $i~=~1,...,d$, we have $\partial_{x_i}g_i = -\frac1d\sum_{j\neq i}\cos(2\pi(x_j-x_i))$ and $\partial_{x_j}g_i = \frac1d\cos(2\pi(x_j-x_i))$ for $j\neq i$, so that every row of $Dg( x)$ has absolute row sum at most $2(d-1)/d$.
        Hence, $\|h\kappa\,Dg( x)\|_\infty \le \frac{2h\kappa(d-1)}{d} < 2h\kappa < 1$ for $h\kappa<\tfrac{1}{2}$, uniformly in $ x$, in the maximum row sum norm.
        Consequently, the lift $ x\mapsto  x+\nu+ h\kappa g( x)$ of $F$ to $\R^d$ is the identity plus a contraction. Therefore, it is bijective on $\R^d$ and the Jacobian $DF = I+h\kappa\,Dg$ is invertible. Since $g$ is $1$-periodic, $F$ is a $C^\infty$-diffeomorphism of $\T^d$.
        Moreover, all eigenvalues $\mu$ of $h\kappa\,Dg( x)$ satisfy $|\mu|\le\|h\kappa\,Dg( x)\|_\infty<2h\kappa$ and thus $|1+\mu|>1-2h\kappa$.  Hence, 
        $\inf_{ x\in\T^d}|\det DF( x)| \ge (1-2h\kappa)^{d}$.

        \begin{remark}\label{rem: step size}
			The condition $h\kappa<\frac12$ is equivalently given by $h<1/(2\kappa)$ and therefore only restricts the step size and not the model.
			This is what one expects, since the flow map of \eqref{eq:kuramoto-ode} is a diffeomorphism for every $\kappa$ and the explicit-Euler map inherits this property as soon as the step size is chosen sufficiently small.
			Moreover $c_0\le(1-2h\kappa)^{-d}\to1$ as $h\to0$.
		\end{remark}
		
		The phase increments $\nu_i$ are drawn i.i.d.\ from $\mathrm{Unif}(0.05,0.15)$ and the coupling per time step is fixed at $h\kappa=0.3$, which satisfies $h\kappa<1/2$.

        As observables we consider tensor-product B-splines of order $m$, given by
        \begin{equation}\label{eq: B-spline}
        	f_m(x) = \prod_{i=1}^d c_m\, N_m(m x_i), \quad x\in\T^d,
        \end{equation}
        where $N_m=N_1*\dots*N_1$ ($m$ factors) with $N_1=\chi_{[0,1)}$ denotes the cardinal B-spline of order $m$. It is supported on $[0,m]$ and $c_m$ is such that $\|f_m\|_{L^2}=1$ for all $d$. 
        In this section we choose $m=2$, and $c_2=\sqrt3$, i.e., $f_2(x)=\sqrt{12}\min\{x,1-x\}$ is the tent function.
        The Fourier coefficients of the $1$-periodic function $N_m(m\,\cdot)$ are $\frac1m\widehat{N_m}(k/m)$ and hence have modulus $\frac1m|\operatorname{sinc}(k/m)|^m$ with $\operatorname{sinc}(\xi)=\sin(\pi\xi)/(\pi\xi)$, see \cite[Section 9.1]{potts2018numerical}. Away from the zeros of the sine function these decay like $|k|^{-m}$ and hence
        \begin{equation}\label{eq: B-spline smoothness}
	       f_m\in H^s \quad\Longleftrightarrow\quad s<m-\tfrac12 .
        \end{equation}
        For $m=2$ this gives $f_2\in H^s$ for all $s<\frac32$ but $f_2\notin H^{3/2}$.
        Since the convergence rate in \Cref{prop:error bound} depends on the regularity of the observable, this example allows us to assess whether the numerically observed rate coincides with the theoretically predicted one.
				
		\begin{remark}\label{rem: smoothness B-spline}
            \Cref{thm: bound interpolation error} and consequently \Cref{prop:error bound} require the observable to have Sobolev smoothness $s>d/2$. For the B-spline $f_2$ in \eqref{eq: B-spline}, this restricts the applicability of the theory to $d\leq 2$. Nevertheless, we present the numerical results up to $d=5$, since the observed convergence rates are unaffected.
			Still, this condition cannot be removed from our worst-case analysis. It ensures both the convergence of the series in \Cref{epstein} and the continuous embedding $H^s \hookrightarrow C$. For $s\le d/2$ point evaluation is unbounded on $H^s$ and hence no algorithm using finitely many  point evaluations can have a finite worst-case error on this space.
            However, this does not contradict the observed rates, since we do not consider a worst-case error over $H^s$, but the error for one particular observable.
            Since our observable \eqref{eq: B-spline}, the dictionary, and the grid $X$ have a tensor-product structure, the interpolation operator also factorises. 
            Hence, the $d$-dimensional interpolation error is governed by the corresponding one-dimensional ones, for which only $s>1/2$ is required.
		\end{remark}
		
		In the following, we describe the implementation of the two Fourier-EDMD approximations $\widehat \calK_1$ and $\widehat \calK_2$ of the Koopman operator and then investigate their convergence behavior as well as the computational cost.
        For this, for a given bandwidth $n$ we use the equispaced tensor-product grid $ X$ defined in~\eqref{eq: grid d-dim}, i.e., the grid contains $N=(2n+1)^d$ grid points and $N$ Fourier frequencies.  
        Both approximations $\widehat \calK_1$ and $\widehat \calK_2$ are built on this grid.
        They only differ in the way the composition of the observable with the dynamics is evaluated, depending on which data are assumed to be available.
		
        For the approximation $\widehat\calK_1 = S_n\calK$ defined in \eqref{eq: one-sided interpolation},the composition $f\circ F$ is evaluated directly on the equispaced grid.
		More precisely, the observable is sampled at image points $F(x_j)$ and a single $d-$dimensional FFT returns the corresponding discrete Fourier coefficients $\widetilde{(f\circ F)}_k$ of the Koopman image. 
		Neither the matrix $L_X$ nor $L_{X^+}$ is ever explicitly formed.
		Since the number of grid points $N=(2n+1)^d$ increases exponentially in $d$, a direct construction of the full $d\times N$ coordinate array would become extremely expensive in higher dimensions. 
		Therefore, we evaluate the observable and the dynamics blockwise along one coordinate direction, using the representation \eqref{eq: Kuramoto coupling sum} in each block. 
		This reduces the size of the temporary coordinate arrays from $\mathcal O(dN)$ to $\mathcal O(dB(2n+1)^{d-1})$, where the block size $B$ is chosen from a prescribed memory budget.
		However, the function values as well as the corresponding Fourier coefficients still require $\mathcal O(N)$ storage.
		
		In contrast to $\widehat\calK_1$, for $\widehat\calK_2=S_n \calK S_n$ defined in \eqref{eq: two-sided interpolation}, only point evaluations $f(x_j)$ of the observable on the uniform grid are available. 
		Hence, we first replace $f$ by its interpolant $S_nf$ and then evaluate the Fourier representation of $S_nf$ directly at image points $F(x_j)$, that is,
        $(S_nf)(F(x_j)) = \sum_{k\in K_n} \widetilde f_k \exp(2\pi i \langle k,F(x_j)\rangle)$.
		This corresponds exactly to the multiplication with the matrix $L_{X^+}$ discussed after \eqref{eq: two-sided interpolation}.
		Since the points $F(x_j)$ are in general not equispaced, the evaluation is performed by a type-2 nonuniform FFT (NFFT/NUFFT), see \cite{potts2018numerical}. 
		The resulting values are then transformed back by a second $d$-dimensional FFT.
		For the NFFT step we use the NFFT3 library \cite{keiner2009nfft} through its Python interface pyNFFT3\footnote{see \url{https://pypi.org/project/pyNFFT3/}}. 
        We set the cut-off $m=7$ of the window function and the oversampling factor $\sigma \ge 2$. For these parameters, the relative error of the NFFT stays below $10^{-12}$, i.e. the approximation error is not influenced by the NFFT.
        Although the NFFT3 library is not restricted to a particular dimension, the approximant $\widehat\calK_2$ is illustrated only for $d\le3$.
        The cost of one NFFT contains a factor $(2m+2)^d$ stemming from the support of the window function. Consequently, for $d\ge4$ the NFFT dominates the overall cost and $\widehat{\calK}_2$ is no longer computationally practical at the bandwidths considered here.
		As for $\widehat\calK_1$, the points $F(x_j)$ are generated blockwise, such that the full $d\times N$ coordinate array is never stored simultaneously.
		In summary, the computation of $\widehat\calK_1$ requires only one FFT and one evaluation of the composition $f \circ F$ on $X$, whereas $\widehat\calK_2$ requires two FFTs, one evaluation of the dynamics $F$ on the grid and one NFFT. 
        
        To quantify the convergence, we consider the relative $L^p$-error
		\begin{equation}\label{eq: relative error}
			\frac{\|\calK f - \widehat\calK_if\|_{L^p}}{\|\calK f \|_{L^p}}
			\quad i=1,2
		\end{equation}
        as we compare the approximation results across different dimensions $d$. 
        We estimate norms via Monte-Carlo quadrature
		%\begin{equation*}
			$\|\calK f - \widehat\calK_if\|_{L^p}
			\approx  \big(\frac{1}{M} \sum_{m=1}^{M} |f(F(z_m)) - \widehat\calK_i f(z_m) |^p\big)^{1/p}
		$ for $p<\infty$
        %\end{equation*} 
        and
        %\begin{equation*}
           $ \|\calK f - \widehat\calK_if\|_{L^\infty}
			\approx \max_{m=1,...,M} |f(F(z_m)) - \widehat\calK_i f(z_m)|$
        %\end{equation*}
		where $z_1,...,z_M$ are drawn independently and uniformly from $\T^d$.
		Hence, the error is evaluated independently of the grid $X$ that was used to construct the Koopman approximation and in particular, does not vanish at the interpolation nodes.
        The evaluation of $\widehat\calK_if$, $i=1,2$ at the $M$ arbitrary points is also performed using a type-2 NFFT, at cost $\calO(N \log N + M)$ with the same NFFT parameters as above.  
        With this, we can choose up to dimension $d=3$ and for all bandwidths $M=10^5$ test points. 
        For $d\ge4$ the error is evaluated directly and $M$ is chosen between $2\cdot10^3$ and $5\cdot10^3$ so that $MN$ stays bounded.

		\Cref{fig:errors} shows the relative $L^2$-errors \eqref{eq: relative error} of the approximants $\widehat \calK_1$ and $\widehat \calK_2$ for the B-spline $f_2$ and different dimensions $d$.
        Since $f_2\in H^s$ for all $s<3/2$ by \eqref{eq: B-spline smoothness}, \Cref{prop:error bound} predicts the rate $n^{-s}$ for every $s<3/2$. In the top row, the error decay in terms of the bandwidth $n$ is illustrated. The five curves are parallel showing that the predicted rate $n^{-3/2}$ is attained independently of $d$ and for both approximants. 
        Notably this includes $d\ge3$, where the assumption $s>d/2$ of \Cref{thm: bound interpolation error} is violated (\Cref{rem: smoothness B-spline}).
        Expressed in terms of the total number $N=(2n+1)^d$ of samples, the same data give the rate $N^{-s/d}$ (bottom row of \Cref{fig:errors}), which is the familiar curse of dimensionality for isotropic Sobolev smoothness.
		
		% shared style for all four panels of both figures
		\pgfplotsset{
			edmdpanel/.style={
				width=0.42\textwidth, height=0.24\textwidth,
				xmode=log, ymode=log,
				grid=both, grid style={gray!22},
				tick align=outside, ylabel near ticks, xlabel near ticks,
				title style={yshift=-0.5ex},
                every axis plot/.append style={line width=0.7pt},
				cycle list={
					{blue,   mark=*,         mark size=1.5pt},
					{red,    mark=square*,   mark size=1.5pt},
					{teal,   mark=triangle*, mark size=1.8pt},
					{orange, mark=diamond*,  mark size=1.8pt},
					{violet, mark=pentagon*, mark size=1.8pt}},
			},
			edmdref/.style={black, densely dashed, line width=0.8pt, no marks, forget plot},
		}
		
		% =============================================================================
		% Figure 1 -- relative L^2 error
        %=============================================================================
		\begin{figure}[htb]
			\centering
			\begin{tikzpicture}
				\begin{groupplot}[
					edmdpanel,
					group style={group size=2 by 2, horizontal sep=1.1cm, vertical sep=1.5cm,
						y descriptions at=edge left},
					xmin=1.5, xmax=7000, ymin=2e-7, ymax=0.6,
					]
					% ---------------------------------------------------------------- row 1: vs n
					\nextgroupplot[
					title={$\widehat{\calK}_1 = S_n\calK$},
					xlabel={bandwidth $n$}, ylabel={relative $L^2$ error},
					legend to name=leg-errors, legend columns=6,
					legend style={font=\footnotesize, draw=none,
						/tikz/every even column/.append style={column sep=6pt}},
					]
					\addplot table[x=n, y=rel1] {bspline2_d1_mac1.dat}; \addlegendentry{$d=1$}
					\addplot table[x=n, y=rel1] {bspline2_d2_mac1.dat}; \addlegendentry{$d=2$}
					\addplot table[x=n, y=rel1] {bspline2_d3_mac1.dat}; \addlegendentry{$d=3$}
					\addplot table[x=n, y=rel1] {bspline2_d4_mac1.dat}; \addlegendentry{$d=4$}
					\addplot table[x=n, y=rel1] {bspline2_d5_mac1.dat}; \addlegendentry{$d=5$}
                    % reference line
					\addplot[edmdref, domain=2:5000, samples=2] {0.2*x^(-1.5)};
					\addlegendimage{black, densely dashed, line width=0.8pt, no marks} \addlegendentry{reference}
					\node[anchor=west, font=\scriptsize] at (axis cs:30,3e-6) {$n^{-3/2}$};
					
					\nextgroupplot[
					title={$\widehat{\calK}_2 = S_n\calK S_n$},
					xlabel={bandwidth $n$},
					]
					\addplot table[x=n, y=rel2] {bspline2_d1_mac1.dat};
					\addplot table[x=n, y=rel2] {bspline2_d2_mac1.dat};
					\addplot table[x=n, y=rel2] {bspline2_d3_mac1.dat};
					\addplot[edmdref, domain=2:5000, samples=2] {0.24*x^(-1.5)};   
					\node[anchor=west, font=\scriptsize] at (axis cs:30,3e-6) {$n^{-3/2}$};
					
					% ---------------------------------------------------------------- row 2: vs N
					\nextgroupplot[
					xlabel={dictionary size $N=(2n+1)^d$}, ylabel={relative $L^2$ error},
					xmin=3, xmax=3e7,
					]
					\addplot table[x=Nd, y=rel1] {bspline2_d1_mac1.dat};
					\addplot table[x=Nd, y=rel1] {bspline2_d2_mac1.dat};
					\addplot table[x=Nd, y=rel1] {bspline2_d3_mac1.dat};
					\addplot table[x=Nd, y=rel1] {bspline2_d4_mac1.dat};
					\addplot table[x=Nd, y=rel1] {bspline2_d5_mac1.dat};
					% one reference N^{-3/(2d)} per curve
					\addplot[edmdref, domain=5:8193,      samples=2] {0.783*x^(-1.5)};
					\addplot[edmdref, domain=25:1.05e6,   samples=2] {1.075*x^(-0.75)};
					\addplot[edmdref, domain=125:2.15e6,  samples=2] {1.361*x^(-0.5)};
					\addplot[edmdref, domain=625:5.76e6,  samples=2] {1.525*x^(-0.375)};
					\addplot[edmdref, domain=3125:9.77e6, samples=2] {1.785*x^(-0.3)};
					\node[anchor=west, font=\scriptsize] at (axis cs:8,3e-6) {$N^{-3/(2d)}$};
					
					\nextgroupplot[
					xlabel={dictionary size $N=(2n+1)^d$},
					xmin=3, xmax=3e7,
					]
					\addplot table[x=Nd, y=rel2] {bspline2_d1_mac1.dat};
					\addplot table[x=Nd, y=rel2] {bspline2_d2_mac1.dat};
					\addplot table[x=Nd, y=rel2] {bspline2_d3_mac1.dat};
					\addplot[edmdref, domain=5:8193,     samples=2] {0.951*x^(-1.5)};
					\addplot[edmdref, domain=25:1.05e6,  samples=2] {1.464*x^(-0.75)};
					\addplot[edmdref, domain=125:2.15e6, samples=2] {1.818*x^(-0.5)};
					\node[anchor=west, font=\scriptsize] at (axis cs:8,3e-6) {$N^{-3/(2d)}$};
				\end{groupplot}
			\end{tikzpicture}
			\\[0.4ex]
			\pgfplotslegendfromname{leg-errors}
			\caption{Relative $L^2$-error of $\widehat\calK_1$ (left) and $\widehat\calK_2$ (right) for the Kuramoto map \eqref{eq:kuramoto-map} and the observable $f_2$ \eqref{eq: B-spline}.
			Top row: in terms of the bandwidth $n$;
			Bottom row: in terms of the number $N=(2n+1)^d$ of samples. 
            }
			\label{fig:errors}
		\end{figure}

        \Cref{tab:rates} makes the comparison of the observed convergence rates with the theoretical ones quantitative. It shows, for each dimension and each approximant, the decay rate of the measured error against the bandwidth $n$.
        For both approximants and all bandwidths, the rates lie between $-1.46$ and $-1.56$ against the predicted $-s$ with $s<3/2$, i.e. the predicted rate is attained up to accuracy of the Monte-Carlo error estimate.
        Moreover, the last column of \Cref{tab:rates} provides information that is not visible in \Cref{fig:errors} due to scaling.
        It gives the range of the ratio $\|\calK f_2-\widehat\calK_2f_2\|_{L^2}/\|\calK f_2-\widehat\calK_1f_2\|_{L^2}$. This ratio stays bounded instead of growing with $n$, and shows in particular that the rates of the two approximants differ only by a constant factor. This is exactly the behavior predicted by \Cref{prop:error bound}, in which $\widehat\calK_1$ and $\widehat\calK_2$ have the same rate and differ only in the constants $C$ and $\widetilde C$.

		\begin{table}[htb]
			\centering
            \caption{Observed convergence rates for the observable $f_2$.}
			\label{tab:rates}
\scalebox{.9}{
			\begin{tabular}{cccccc}
				\hline
				$d$ & bandwidth & $N_{\max}$
				& rate $\widehat{\calK}_1$ & rate $\widehat{\calK}_2$
				& error ratio \\
				\hline
				$1$ & $8\le n\le 4096$ & $8\,193$      & $-1.49$ & $-1.50$ & $1.01$--$1.40$ \\
				$2$ & $8\le n\le 512$  & $1\,050\,625$ & $-1.49$ & $-1.48$ & $1.27$--$1.36$ \\
				$3$ & $8\le n\le 64$   & $2\,146\,689$ & $-1.46$ & $-1.48$ & $1.33$--$1.42$ \\
				$4$ & $8\le n\le 24$   & $5\,764\,801$ & $-1.48$ & --      & --                   \\
				$5$ & $4\le n\le 12$   & $9\,765\,625$ & $-1.56$ & --      & --                   \\
				\hline
			\end{tabular}}
		\end{table}

        While \Cref{prop:error bound} provides $L^p$-error bounds for all $2 \le p \le \infty$, so far the error has been measured in $L^2$ only. 
        \Cref{fig:lp errors} shows for the observable $f_2$ and $d=1$ the relative $L^p$-error of both approximants for $p=2,4,\infty$.
        % , where the $L^p$-norms are estimated by Monte-Carlo quadrature on the same test set as above and the $L^\infty$-norm by the maximum over the test points.
        \Cref{prop:error bound} predicts in terms of the bandwidth the rate $n^{-s+d(\frac12-\frac1p)}$.
        With $d=1$ and $s = 3/2$, this is $n^{-\frac32}$ for $p=2$, $n^{-\frac54}$ for $p=4$ and $n^{-1}$ for $p=\infty$. 
        The measured rates are $-1.49$, $-1.23$ and $-0.99$ for the approximant $\widehat\calK_1$ and agree with those of $\widehat\calK_2$ within $0.02$.
        % $-1.50$, $-1.25$ and $-1.01$ for $\widehat\calK_2$. 
        For $d\ge2$, as described in \Cref{rem: smoothness B-spline}, the observable, the dictionary and the grid factorise, such that the $d$-dimensional $L^p$-error is again governed by the one-dimensional ones.
        The observed loss in the exponent stays around $1/2-1/p$ instead of growing to $d(1/2-1/p)$. 
        Therefore, \Cref{fig:lp errors} only shows $d=1$.

        \begin{figure}[htb]\centering
        \begin{tikzpicture}
            \begin{groupplot}[edmdpanel,  group style={group size=2 by 1, horizontal sep=1.1cm, y descriptions at=edge left}, xmin=1.5, xmax=6e4, ymin=2e-7, ymax=0.6]
            \nextgroupplot[title={$\widehat{\calK}_1 = S_n\calK$}, xlabel={bandwidth $n$}, ylabel={relative $L^p$ error},legend to name=leg-lp, legend columns=4,legend style={font=\footnotesize, draw=none,/tikz/every even column/.append style={column sep=6pt}}]
            
                \addplot table[x=n, y=rel1p2]   {lp_bspline2_d1.dat}; \addlegendentry{$p=2$}
                \addplot table[x=n, y=rel1p4]   {lp_bspline2_d1.dat}; \addlegendentry{$p=4$}
                \addplot table[x=n, y=rel1pinf] {lp_bspline2_d1.dat}; \addlegendentry{$p=\infty$}
                \addplot[edmdref, domain=16:5000, samples=2] {0.200*x^(-1.5)};
                \addplot[edmdref, domain=16:5000, samples=2] {0.168*x^(-1.25)};
                \addplot[edmdref, domain=16:5000, samples=2] {0.203*x^(-1.0)};
                \addlegendimage{black, densely dashed, line width=0.8pt, no marks} \addlegendentry{reference}

                \node[anchor=west, font=\scriptsize] at (axis cs:5500,1.5e-6) {$n^{-3/2}$};
                \node[anchor=west, font=\scriptsize] at (axis cs:5500,1.5e-5) {$n^{-5/4}$};
                \node[anchor=west, font=\scriptsize] at (axis cs:5500,1.5e-4) {$n^{-1}$};

            \nextgroupplot[title={$\widehat{\calK}_2 = S_n\calK S_n$}, xlabel={bandwidth $n$}]
                
                \addplot table[x=n, y=rel2p2]   {lp_bspline2_d1.dat};
                \addplot table[x=n, y=rel2p4]   {lp_bspline2_d1.dat};
                \addplot table[x=n, y=rel2pinf] {lp_bspline2_d1.dat};
                \addplot[edmdref, domain=16:5000, samples=2] {0.243*x^(-1.5)};
                \addplot[edmdref, domain=16:5000, samples=2] {0.194*x^(-1.25)};
                \addplot[edmdref, domain=16:5000, samples=2] {0.223*x^(-1.0)};
                
                \node[anchor=west, font=\scriptsize] at (axis cs:5500,1.5e-6) {$n^{-3/2}$};
                \node[anchor=west, font=\scriptsize] at (axis cs:5500,1.5e-5) {$n^{-5/4}$};
                \node[anchor=west, font=\scriptsize] at (axis cs:5500,1.5e-4) {$n^{-1}$};
                
            \end{groupplot}
        \end{tikzpicture}\\[0.4ex]
        \pgfplotslegendfromname{leg-lp}
        \caption{Relative $L^p$-error for $d=1$ and $p=2,4,\infty$ of $\widehat\calK_1$ (left) and $\widehat\calK_2$ (right) for the observable $f_2$ and the Kuramoto map \eqref{eq:kuramoto-map}.
        }
        \label{fig:lp errors}
        \end{figure}

        Finally, \Cref{fig:times} shows the time needed to construct the two approximants of \Cref{fig:errors}, in terms of the bandwidth $n$ (top) and the dictionary size $N$ (bottom). By the discrete Fourier structure of $L_X$, see \Cref{sec: Fourier EDMD}, both approximants are expected to cost $\calO(N\log N)$. The reference lines $C\,N\log N$ are drawn over the range in which the runtime follows the asymptotic behavior, in the top row with one constant per dimension and in the bottom row with a single constant for all dimensions. For $\widehat\calK_1$, this confirms that the cost is essentially governed by $N$ rather than by $d$.
        For $\widehat\calK_2$ the constant is larger due to the NFFT and the additional FFT, and at the largest bandwidth it grows by a factor of four from $d=2$ to $d=3$ through the window function of the NFFT.
        At the largest bandwidth of each $d\le 3$, building $\widehat\calK_2$ is $3.5$ to $10.3$ times more expensive than $\widehat\calK_1$.
        The largest dictionary, $d=5$ and $n=12$, uses $N = 25^5 = 9\,765\,625$ Fourier modes and takes $1.4$~s.
        The corresponding EDMD matrix would have $9.5\cdot10^{13}$ entries, i.e., about $1.5$ PB in complex double precision, whereas the matrix-free evaluation stores a coefficient array of only $156$ MB. Thus, besides the gain in computation time by the FFT, the storage requirement is the decisive factor here, since such an EDMD matrix cannot even be stored.
        % That this remains feasible in higher dimensions is due to the blockwise evaluation of the dynamics and the observable, which never stores the full $d\times N$ coordinate array.

        % =============================================================================
		% Figure 2 -- computation time
		% =============================================================================
		\begin{figure}[htb]
			\centering
			\begin{tikzpicture}
				\begin{groupplot}[
					edmdpanel,
					group style={group size=2 by 2, horizontal sep=1.1cm, vertical sep=1.5cm,
						y descriptions at=edge left},
					xmin=1.5, xmax=7000, ymin=2e-5, ymax=40,
					]
					% ---------------------------------------------------------------- row 1: vs n
					\nextgroupplot[
					title={$\widehat{\calK}_1 = S_n\calK$},
					xlabel={bandwidth $n$}, ylabel={computation time [s]},
					legend to name=leg-times, legend columns=6,
					legend style={font=\footnotesize, draw=none,
						/tikz/every even column/.append style={column sep=6pt}},
					]
					\addplot table[x=n, y=time1] {bspline2_d1_mac1.dat}; \addlegendentry{$d=1$}
					\addplot table[x=n, y=time1] {bspline2_d2_mac1.dat}; \addlegendentry{$d=2$}
					\addplot table[x=n, y=time1] {bspline2_d3_mac1.dat}; \addlegendentry{$d=3$}
					\addplot table[x=n, y=time1] {bspline2_d4_mac1.dat}; \addlegendentry{$d=4$}
					\addplot table[x=n, y=time1] {bspline2_d5_mac1.dat}; \addlegendentry{$d=5$}

					\addplot[edmdref, domain=512:4096, samples=20]{1.31e-8*(2*x+1)^1*ln((2*x+1)^1)};
					\addplot[edmdref, domain=64:512,   samples=20]{5.28e-9*(2*x+1)^2*ln((2*x+1)^2)};
					\addplot[edmdref, domain=24:64,    samples=20]{6.90e-9*(2*x+1)^3*ln((2*x+1)^3)};
					\addplot[edmdref, domain=8:24,     samples=20]{8.60e-9*(2*x+1)^4*ln((2*x+1)^4)};
					\addplot[edmdref, domain=2:12,     samples=20]{1.43e-8*(2*x+1)^5*ln((2*x+1)^5)};
					\addlegendimage{black, densely dashed, line width=0.8pt, no marks} \addlegendentry{reference}
					
					\nextgroupplot[
					title={$\widehat{\calK}_2 = S_n\calK S_n$},
					xlabel={bandwidth $n$},
					]
					\addplot table[x=n, y=time2] {bspline2_d1_mac1.dat};
					\addplot table[x=n, y=time2] {bspline2_d2_mac1.dat};
					\addplot table[x=n, y=time2] {bspline2_d3_mac1.dat};
					\addplot[edmdref, domain=64:512, samples=20]{2.06e-8*(2*x+1)^2*ln((2*x+1)^2)};
					\addplot[edmdref, domain=24:64,  samples=20]{7.00e-8*(2*x+1)^3*ln((2*x+1)^3)};
					
					% ---------------------------------------------------------------- row 2: vs N
					\nextgroupplot[
					xlabel={dictionary size $N=(2n+1)^d$}, ylabel={computation time [s]},
					xmin=3, xmax=3e7,
					]
					\addplot table[x=Nd, y=time1] {bspline2_d1_mac1.dat};
					\addplot table[x=Nd, y=time1] {bspline2_d2_mac1.dat};
					\addplot table[x=Nd, y=time1] {bspline2_d3_mac1.dat};
					\addplot table[x=Nd, y=time1] {bspline2_d4_mac1.dat};
					\addplot table[x=Nd, y=time1] {bspline2_d5_mac1.dat};
					% reference line
					\addplot[edmdref, domain=1e4:1e7, samples=30] {7.69e-9*x*ln(x)};
					\node[anchor=north west, font=\scriptsize] at (axis cs:2e5,1.5e-2) {$N\log N$};
					
					\nextgroupplot[
					xlabel={dictionary size $N=(2n+1)^d$},
					xmin=3, xmax=3e7,
					]
					\addplot table[x=Nd, y=time2] {bspline2_d1_mac1.dat};
					\addplot table[x=Nd, y=time2] {bspline2_d2_mac1.dat};
					\addplot table[x=Nd, y=time2] {bspline2_d3_mac1.dat};
					\addplot[edmdref, domain=1e4:3e6, samples=30] {4.75e-8*x*ln(x)};
					\node[anchor=north west, font=\scriptsize] at (axis cs:2e5,6e-2) {$N\log N$};
				\end{groupplot}
			\end{tikzpicture}
			\\[0.4ex]
			\pgfplotslegendfromname{leg-times}
			\caption{Computation time for constructing the two Koopman approximants $\widehat\calK_1$ (left) and $\widehat\calK_2$ (right), excluding the error estimation; same data, colors and layout as in \Cref{fig:errors}.}
            \label{fig:times}
		\end{figure}

		\subsection{The chaotic ABC flow on $\mathds{T}^3$}\label{sec: ABC}

		In this section, we consider a second example, in which the map $F$ is the flow map of a chaotic ordinary differential equation in which the Koopman image of an observable develops fine-scale structures. 

		We consider the ordinary differential equation
		\begin{equation}\label{eq: ABC ODE}
			\dot x(t) = v(x(t)), \qquad x(0)=x_0\in\T^3,
		\end{equation}
		with the velocity field of the Arnold--Beltrami--Childress (ABC) flow in the chaotic regime, see \cite{dombre1986chaotic}, where all three parameters of the flow are equal to one, i.e.,
		\begin{equation}\label{eq: ABC}
			v(x) = \begin{pmatrix} \sin(2\pi x_3)+\cos(2\pi x_2)\\ \sin(2\pi x_1)+\cos(2\pi x_3)\\ \sin(2\pi x_2)+\cos(2\pi x_1)\end{pmatrix}, \qquad x\in\T^3 .
		\end{equation}
		Since $v$ is smooth and $1$-periodic, and hence globally Lipschitz, \eqref{eq: ABC ODE} has a unique global solution for every $x_0\in\T^3$ and generates the flow $\Phi_t:\T^3\to\T^3$, $t\in\R$. 
        For a sampling time $\Delta t>0$, we are interested in the Koopman operator of the time-$\Delta t$ map $\Phi_{\Delta t}$, which 
        % However, since $\Phi_{\Delta t}$ 
        is not available in closed form. Therefore, we take
        \begin{equation}\label{eq: ABC map}
        	F_h \approx \Phi_{\Delta t}, \qquad (\calK f)(x) = f(F_h(x)),
        \end{equation}
        where $F_h$ is the discretized flow map \eqref{eq: ABC flow map} obtained from an order-four splitting method with step size at most $h=0.01$, which exploits that $v$ is the sum of three shear fields with explicit flows, see \Cref{app: ABC flow map}.
        We consider the sampling times $\Delta t\in\,\{0.125,0.25,0.5,1\}$.

		We briefly verify that the assumptions of \Cref{prop:error bound} are satisfied.
        Since no component of $v$ depends on its own variable, we have $\operatorname{div}v=0$ and Liouville's formula yields $\det D\Phi_{\Delta t}\equiv1$, that is, the exact flow is volume preserving. 
        The discretized flow map \eqref{eq: ABC flow map} inherits this property. Since each sub-flow has unit Jacobian determinant, see \Cref{app: ABC flow map}, $F_h$ is as a composition of these sub-flows a $C^\infty$-diffeomorphism of $\T^3$ with $\det DF_h\equiv1$.
        Hence, $c_0=1$ and the change of variables $y=F_h(x)$ shows that the Koopman operator is an isometry on $L^p$ for all $1\le p\le\infty$, such that $\|\calK f\|_{L^p}=\|f\|_{L^p}$ in the relative error \eqref{eq: relative error}.
        In contrast to $c_0$, the constant $C_F$ in \Cref{thm:koopman-well-defined and bound} depends on the derivatives of $F_h$, which grow with $\Delta t$ due to the stretching by the flow, as denoted in the top of \Cref{fig:abc-cube}. 

        As observable, we choose the B-spline $f_4$ in \eqref{eq: B-spline} with $c_4=\sqrt{1260/151}$. By \eqref{eq: B-spline smoothness}, $f_4\in H^s$ if and only if $s<7/2$. In contrast to $f_2$, which fails to satisfy the assumption $s>d/2$ for $d=3$, cf.~\Cref{rem: smoothness B-spline}, \Cref{prop:error bound} now applies with $\sigma=4$, predicting the $L^2$-error rate arbitrarily close to $n^{-7/2}$.

        \Cref{fig:abc-cube} illustrates the dynamics. 
        Throughout this section we represent $\T^3$ by the cube $[-\frac12,\frac12]^3$ rather than by $[0,1)^3$, so that the three faces shown are the planes $x_i=\pm\frac12$, on which the peak of $f_4$ lies. 
        Since $\calK f=f\circ F_h$ the level sets of $\calK f$ are the preimages of the level sets of $f$ under $F_h$, i.e., they consist of the initial positions of all trajectories that are located in the respective level set of $f$ at time $\Delta t$. 
        The top row shows the auxiliary observable $f(x)=\sin(2\pi x_3)$, which is used for this illustration only. Since it depends on $x_3$ alone, its level sets are the horizontal planes $x_3=\mathrm{const}$, and their transformation by the flow can be seen directly. 
        The bottom row shows $f_4$, whose peak $(\frac12,\frac12,\frac12)$ lies at the corners of the cube.
        In both cases the initially smooth level sets are folded and stretched by the chaotic dynamics, and $\calK f$ develops thin sheets and filaments whose width decreases rapidly with $\Delta t$. 
        This is quantified by the spectral norm of the Jacobian $DF_h$ of the discretized flow map, which measures the local stretching. Its maximum over 2000 uniformly distributed points, given above each column, grows from $9.1$ for $\Delta t=0.25$ to $2.05\cdot10^3$ for $\Delta t=1$, such that finer grids are needed to resolve $\calK f$ for larger $\Delta t$.

		% \begin{figure}[htb]
		% 	\centering
		% 	\includegraphics[width=\textwidth]{abc_cube.pdf}
		% 	\caption{Koopman image $\calK f=f\circ F_h$ for the ABC map \eqref{eq: ABC map}, evaluated directly on the faces $x_1=\frac12$, $x_2=-\frac12$ and $x_3=\frac12$ of the cube $[-\frac12,\frac12]^3$. Top row: $f(x)=\sin(2\pi x_3)$; bottom row: B-spline $f_4$ \eqref{eq: B-spline}. 
  %           % The numbers above the columns give the maximum of $\|DF_h(x)\|_2$ over $2000$ uniformly distributed points.
  %           }
		% 	\label{fig:abc-cube}
		% \end{figure}

        \begin{figure}[htb]
        	\centering
        	\includegraphics[width=\textwidth,trim={0 117 0 5},clip]{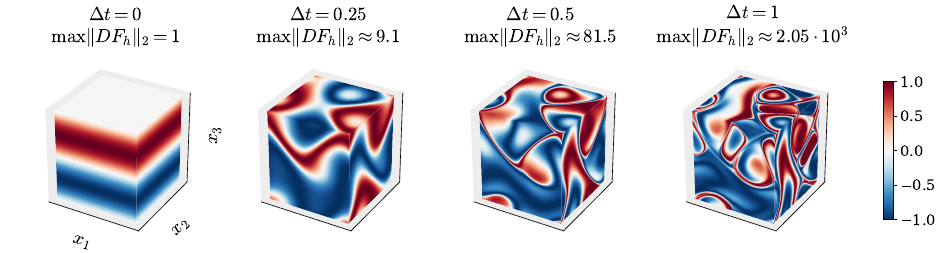}\\[-2pt]
        	\includegraphics[width=\textwidth,trim={0 7 0 145},clip]{abc_cube.pdf}
        	\caption{Koopman image $\calK f=f\circ F_h$ for the ABC map \eqref{eq: ABC map}, evaluated directly on the faces $x_1=\frac12$, $x_2=-\frac12$ and $x_3=\frac12$ of the cube $[-\frac12,\frac12]^3$. Top row: $f(x)=\sin(2\pi x_3)$; bottom row: B-spline $f_4$ \eqref{eq: B-spline}. }
        	\label{fig:abc-cube}
        \end{figure}

		We use the bandwidths $2\le n\le 64$, i.e., dictionaries of up to $N=129^3\approx2.1\cdot10^6$ Fourier modes. 
        We estimate the relative errors \eqref{eq: relative error} for all $n$ and both approximants by Monte-Carlo quadrature with $M=10^5$ uniformly distributed test points and evaluate $\widehat\calK_i f$, $i=1,2$, at these points by a type-2 NFFT analogous to \Cref{sec: Kuramoto}.
        
		The top row of \Cref{fig:abc} shows the relative $L^2$-errors of both Koopman approximants with the observable $f_4$ \eqref{eq: B-spline} in terms of the bandwidth $n$. 
        For $\Delta t=0.125$, the error of $\widehat\calK_1$ decays with rate $n^{-72}$ for all bandwidths, which coincides with the rate $n^{-s}$, $s<7/2$, predicted by \Cref{prop:error bound}. 
        With increasing $\Delta t$, the error curves are shifted upwards and the asymptotic regime starts at larger bandwidths.
        While for $\Delta t=0.25$ the local slopes reach $-7/2$ from $n=32$ on, for $\Delta t=1$ they stay near $-1/2$, i.e., the asymptotic regime is not reached at all for the bandwidths considered.
        This is in line with \Cref{prop:error bound}, whose constant grows with $C_F$ and thus with the stretching by the flow, cf.~\Cref{fig:abc-cube}.
        The relative $L^2$-errors of $\widehat\calK_2$ are almost indistinguishable from those of $\widehat\calK_1 $. Since $\widehat\calK_1 f_4-\widehat\calK_2 f_4=S_n\calK(f_4-S_nf_4)$, the difference between the two approximants is determined by the interpolation error of $f_4$, which does not depend on $\Delta t$. 
        The errors of $\widehat\calK_1$ and $\widehat\calK_2$, on the other hand, are governed by the interpolation error of $f_4\circ F_h$ and therefore grow with the sampling time. Only for $\Delta t=0.125$, where the deformation of $f_4$ by $F_h$ is moderate, the interpolation error of $f_4$ is noticeable and $\widehat\calK_2$ is about $20\%$ less accurate.
        For $\Delta t\ge0.25$ 
        % the error is dominated by the interpolation error of $f_4\circ F_h$ such that 
        the two approximants agree to within $0.4\%$ for every $n\ge8$.
        
		The bottom row of \Cref{fig:abc} shows the computation time for constructing the two approximants.  
        The cost of $\widehat\calK_1$ is dominated by the evaluation of $F_h$ at $N$ grid points rather than by the transform. For $n=64$ and $\Delta t=1$, constructing $\widehat\calK_1$ takes $25.9$~s, of which the FFT accounts for $0.02$~s. 
        Therefore, the measured times grow essentially linear in $N$, which stays within the quasi-linear reference. 
        Since $\widehat\calK_2$ evaluates the flow at the same grid points and adds one NFFT and one FFT, the additional cost is governed by these two transforms and hence independent of $\Delta t$.
        The cost of $\widehat\calK_1$, on the other hand, grows 
        % from $3.2$~s to $25.9$~s 
        with the sampling time. Hence, the ratio between $\widehat\calK_2$ and $\widehat\calK_1$ decreases from $1.56$ at $\Delta t=0.125$ to $1.06$ at $\Delta t=1$. 
        For the Kuramoto map in \Cref{sec: Kuramoto}, whose closed-form dynamics is cheap to evaluate, the ratio was between $3.5$ and $10.3$ at the largest bandwidths. Hence, the fully data-driven approximant is cheapest, relative to $\widehat\calK_1$, when evaluating the dynamics is expensive.

\begin{figure}[htb]
	\centering
	\begin{tikzpicture}
		\begin{groupplot}[
			edmdpanel,
			group style={group size=2 by 2, horizontal sep=1.1cm, vertical sep=1.6cm,
				y descriptions at=edge left},
			]
			\nextgroupplot[
			title={$\widehat\calK_1 = S_n\calK$},
			xmin=1.5, xmax=100, ymin=2e-7, ymax=2,
			xlabel={bandwidth $n$}, ylabel={relative $L^2$ error},
			legend to name=leg-abc, legend columns=6,
			legend style={font=\footnotesize, draw=none,
				/tikz/every even column/.append style={column sep=6pt}},
			]
			\addplot table[x=n, y=rel1] {abc_bspline4_T0p125.dat}; \addlegendentry{$\Delta t=0.125$}
			\addplot table[x=n, y=rel1] {abc_bspline4_T0p25.dat};  \addlegendentry{$\Delta t=0.25$}
			\addplot table[x=n, y=rel1] {abc_bspline4_T0p5.dat};   \addlegendentry{$\Delta t=0.5$}
			\addplot table[x=n, y=rel1] {abc_bspline4_T1.dat};     \addlegendentry{$\Delta t=1$}
			\addplot[edmdref, domain=8:64, samples=2] {1.0*x^(-3.5)};
			\addlegendimage{gray, mark=x, mark size=2pt}
			\addlegendentry{FFT / NFFT only}
			\addlegendimage{black, densely dashed, line width=0.8pt, no marks}
			\addlegendentry{ref.}
			\node[anchor=east, font=\scriptsize] at (axis cs:30,3e-6) {$n^{-7/2}$};

			\nextgroupplot[
			title={$\widehat\calK_2 = S_n\calK S_n$},
			xmin=1.5, xmax=100, ymin=2e-7, ymax=2,
			xlabel={bandwidth $n$},
			]
			\addplot table[x=n, y=rel2] {abc_bspline4_T0p125.dat};
			\addplot table[x=n, y=rel2] {abc_bspline4_T0p25.dat};
			\addplot table[x=n, y=rel2] {abc_bspline4_T0p5.dat};
			\addplot table[x=n, y=rel2] {abc_bspline4_T1.dat};
			\addplot[edmdref, domain=8:64, samples=2] {1.2*x^(-3.5)};
			\node[anchor=east, font=\scriptsize] at (axis cs:30,3e-6) {$n^{-7/2}$};

			\nextgroupplot[
			xmin=50, xmax=5e6, ymin=1e-5, ymax=60,
			xlabel={dictionary size $N=(2n+1)^3$}, ylabel={computation time [s]},
			]
			\addplot table[x=Nd, y=time1] {abc_bspline4_T0p125.dat};
			\addplot table[x=Nd, y=time1] {abc_bspline4_T0p25.dat};
			\addplot table[x=Nd, y=time1] {abc_bspline4_T0p5.dat};
			\addplot table[x=Nd, y=time1] {abc_bspline4_T1.dat};
			\addplot[gray, mark=x, mark size=2pt, forget plot]
				table[x=Nd, y=tfft] {abc_bspline4_T1.dat};
			\addplot[edmdref, domain=1e4:3e6, samples=30] {1.7e-6*x*ln(x)};
			\node[anchor=east, font=\scriptsize] at (axis cs:5e4,4.2) {$N\log N$};

			\nextgroupplot[
			xmin=50, xmax=5e6, ymin=1e-5, ymax=60,
			xlabel={dictionary size $N=(2n+1)^3$},
			]
			\addplot table[x=Nd, y=time2] {abc_bspline4_T0p125.dat};
			\addplot table[x=Nd, y=time2] {abc_bspline4_T0p25.dat};
			\addplot table[x=Nd, y=time2] {abc_bspline4_T0p5.dat};
			\addplot table[x=Nd, y=time2] {abc_bspline4_T1.dat};
			\addplot[gray, mark=x, mark size=2pt, forget plot]
				table[x=Nd, y=tnfft] {abc_bspline4_T1.dat};
			\addplot[edmdref, domain=1e4:3e6, samples=30] {1.8e-6*x*ln(x)};
			\node[anchor=east, font=\scriptsize] at (axis cs:5e4,4.4) {$N\log N$};
		\end{groupplot}
	\end{tikzpicture}
	\\[0.4ex]
	\pgfplotslegendfromname{leg-abc}
	\caption{Top row: relative $L^2$-error of $\widehat{\calK}_1$ (left) and $\widehat{\calK}_2$ (right) for the ABC map \eqref{eq: ABC map} and the observable $f_4$ \eqref{eq: B-spline}. Bottom row: computation time for constructing $\widehat{\calK}_1$ (left) and $\widehat{\calK}_2$ (right), excluding the error estimation.}
	\label{fig:abc}
\end{figure}

		\section{Conclusion}\label{sec: conclusion}
		We have analyzed EDMD with a Fourier dictionary on the torus for data on an equispaced tensor grid. Since the Fourier sampling matrix is unitary up to scaling, the EDMD least-squares problem admits a unique closed-form solution, and Fourier-EDMD coincides with trigonometric interpolation of the Koopman image. Based on this identification, we derived $L^p$-error bounds, $2 \le p \le \infty$, of optimal order $(n^d)^{-\frac sd+\frac 12-\frac 1p}$ in the bandwidth $n$ for observables of Sobolev smoothness $s>d/2$, both with measurements of the observable along the dynamics and in the fully data-driven setting, with constants explicit in $F$, its smoothness $s$, and the dimension $d$. Along the way, we established an error estimate for trigonometric interpolation with explicit constants, boundedness of the Koopman operator on Fourier-based Sobolev spaces of fractional order, and the isometric identification of these spaces with interpolation spaces. Moreover, the same structure yields a matrix-free implementation by means of the FFT and, for the nonequispaced image points, the NFFT at quasi-linear cost in the dictionary size. For the Kuramoto model on $\T^d$ up to $d=5$ with dictionaries of roughly $10^7$ modes, we observed the predicted rates, for $d\ge3$ even beyond the range $s>d/2$ covered by our analysis, and quasi-linear cost. For the time-$\Delta t$ map of the chaotic ABC flow on $\T^3$, we observed the predicted rate for a B-spline of order four as long as the stretching by the flow is moderate, whereas for larger sampling times the constant, which grows with the derivatives of the flow map, dominates for the bandwidths considered.

         % =============================================================================
		% Appendix -- evaluation of the ABC flow map
		% =============================================================================
		\appendix
        \crefalias{section}{appendix}
        \crefalias{subsection}{appendix}

		\section{Evaluation of the ABC flow map}\label{app: ABC flow map}
		The velocity field \eqref{eq: ABC} is the sum $v=v^{(1)}+v^{(2)}+v^{(3)}$ of the shear fields
		\begin{equation*}
			v^{(1)}(x) = \begin{pmatrix}\sin(2\pi x_3)\\ \cos(2\pi x_3)\\ 0\end{pmatrix},\qquad
			v^{(2)}(x) = \begin{pmatrix}0\\ \sin(2\pi x_1)\\ \cos(2\pi x_1)\end{pmatrix},\qquad
			v^{(3)}(x) = \begin{pmatrix}\cos(2\pi x_2)\\ 0\\ \sin(2\pi x_2)\end{pmatrix}.
		\end{equation*}
		Each $v^{(i)}$ collects the two terms of \eqref{eq: ABC} involving $x_{\sigma(i)}$ with $\sigma=(3,1,2)$. Since no component of $v$ depends on its own variable, the $\sigma(i)$-th component of $v^{(i)}$ vanishes, i.e., each $v^{(i)}$ depends only on one variable $x_{\sigma(i)}$, which stays unchanged.
        Hence, its flow is explicitly given by $\varphi^{(i)}_\tau(x) = x+\tau v^{(i)}(x)$, $\tau\in\R$. We combine these flows by the Strang splitting $S_2(\tau) = \varphi^{(1)}_{\tau/2}\circ\varphi^{(2)}_{\tau/2}\circ\varphi^{(3)}_{\tau}\circ\varphi^{(2)}_{\tau/2}\circ\varphi^{(1)}_{\tau/2}$ and raise the order to four by the symmetric triple jump composition $S_4(\tau) = S_2(\gamma_1\tau)\circ S_2(\gamma_0\tau)\circ S_2(\gamma_1\tau)$ with $\gamma_1=(2-2^{1/3})^{-1}$ and $\gamma_0=1-2\gamma_1$, , which are determined by the consistency condition $2\gamma_1+\gamma_0=1$, and by the cancellation of the third-order term, $2\gamma_1^3+\gamma_0^3=0$, see \cite{yoshida1990construction} and \cite[Sections II.4 and II.5]{hairer2006geometric}. 
        With $q=\lceil \Delta t/h\rceil$ steps, this yields the map
		% \begin{equation}\label{eq: ABC flow map}
		% 	F_h = \underbrace{S_4(\Delta t/q)\circ\cdots\circ S_4(\Delta t/q)}_{q\text{ times}} \approx \Phi_{\Delta t},
		% \end{equation}
        \begin{equation}\label{eq: ABC flow map}
			F_h = S_4(\Delta t/q)\circ\cdots\circ S_4(\Delta t/q) \approx \Phi_{\Delta t},
        \end{equation}
		whose evaluation requires $12q+1$ sub-flows per point, since consecutive sub-flows of the same field can be merged. As the step size does not depend on the initial value, $F_h$ is one fixed map, which is evaluated both at the grid points and at the test points. Compared with a reference solution of order six with step size $h/8$, the root mean square distance between $F_h$ and $\Phi_{\Delta t}$ over $2000$ uniformly distributed points ranges from $7.7\cdot10^{-7}$ to $2.0\cdot10^{-5}$ depending on $\Delta t$.
        Each sub-flow $\varphi^{(i)}_\tau$ is a smooth map of $\T^3$ with inverse $\varphi^{(i)}_{-\tau}$.
		Moreover, $D\varphi^{(i)}_\tau = I+\tau\,Dv^{(i)}$, and since $v^{(i)}$ depends only on $x_{\sigma(i)}$, 
        % $Dv^{(i)}$ has the single nonzero column $w\coloneqq\partial_{x_{\sigma(i)}}v^{(i)}$, that is,
		we have $D\varphi^{(i)}_\tau = I+\tau\,w\,e_{\sigma(i)}^\top$ with $w\coloneqq\partial_{x_{\sigma(i)}}v^{(i)}$. The matrix determinant lemma yields $\det D\varphi^{(i)}_\tau = 1+\tau\,e_{\sigma(i)}^\top w = 1+\tau\,w_{\sigma(i)} = 1$, since $v^{(i)}_{\sigma(i)}\equiv0$ and hence $w_{\sigma(i)}=~0$. As a composition of such sub-flows, $F_h$ is a $C^\infty$-diffeomorphism of $\T^3$ with $\det DF_h\equiv1$, for every step size.
        % and every order of the composition.

        \section*{Declarations}

        During the preparation of this work the authors used a large language model (Claude Opus 5.5; Anthropic) to improve the language and grammar of the manuscript and to identify suitable benchmarks such as the ABC flow of \Cref{sec: ABC}. The authors assume responsibility for all content.

		\bibliographystyle{abbrv}
		\bibliography{references}
		
	\end{document}